\documentclass[11pt, oneside]{article} % use "amsart" for AMSLaTeX format

\usepackage[utf8]{inputenc}
\usepackage[T1]{fontenc}
\usepackage{xcolor} % to expand list of colours available: black, blue, brown, cyan, darkgrey, 
\usepackage{color} % to write in colours 
\usepackage{amsmath}
\usepackage{amssymb} % note: it loads \usepackage{amsfonts}
\usepackage{mathrsfs} % for italic letters, use the command: \mathscr{}
\usepackage{bm} % for isolated bold symbols within an equation
\usepackage{bbm} % for numbers in the \mathbb style
\usepackage{float}
\usepackage{multirow}
\usepackage{comment}

\usepackage{nameref}
\usepackage{xr-hyper}
\usepackage{makecell, float}
\usepackage{amsthm} 
\usepackage{amsmath}
\theoremstyle{plain}
\newtheorem{theorem}{Theorem}[section] % use [section] (or [chapter] in book class) to follow the section (chapter) number
\newtheorem{condition}[theorem]{Condition}
\newtheorem{example}[theorem]{Example}
\newtheorem{lemma}[theorem]{Lemma}

\newtheorem{remark}[theorem]{Remark}

\usepackage{algorithm2e} % to write algorithms
\SetKwInput{KwInputs}{Inputs}
\usepackage[noend]{algpseudocode}
\makeatletter
\def\BState{\State\hskip-\ALG@thistlm}
\makeatother

\usepackage{graphicx} % to insert images. Use the command: \includegraphic{}
\usepackage{subcaption} % to add captions to subfigures
\graphicspath{{figures/}}

\usepackage{tikz} % to create plots with LaTeX
\usetikzlibrary{positioning ,shadows,matrix}
\usepackage{pgfplots}
\pgfplotsset{compat=1.18}
\usepgfplotslibrary{groupplots}
\usepackage[english]{babel} % to manage the language. It is possible to specify more than one: e.g.[english,italian], and switch among them. It automatically handles the division in syllables

\usepackage[]{natbib}
\usepackage{bookmark}

\IfFileExists{xurl.sty}{\usepackage{xurl}}{} % add URL line breaks if available
\hypersetup{
  pdftitle={Title},
  pdfauthor={Author 1; Author 2},
  pdfkeywords={3 to 6 keywords, that do not appear in the title},
  colorlinks=true,
  linkcolor={blue},
  filecolor={Maroon},
  citecolor={Blue},
  urlcolor={Blue},
  pdfcreator={LaTeX via pandoc}}

\usepackage{hyperref} % to include hyperlinks
\hypersetup{
    pdfpagemode={UseOutlines},
    bookmarksopen,
    pdfstartview={FitH},
    colorlinks,
    linkcolor={blue}, % color of the index
    citecolor={blue}, % color of bibliographic references
    urlcolor={blue} % color of urls
}

\usepackage[letterpaper,margin=1.25in]{geometry} % to set the margins

\usepackage{enumitem} % to set the enumerate/itemize symbols. Use, e.g., 
\def\iid{\overset{\textnormal{iid}}{\sim}} % i.i.d. symbol
\DeclareMathOperator{\sgn}{sgn}

\makeatletter  % cartesian product symbol
\@ifundefined{@currsizeindex} 
  {\RequirePackage{relsize}\let\dolarger\relsize} 
  {\def\dolarger#1{\larger[#1]}} 
\newcommand*\@@bigtimes[2]{\vphantom{\prod} 
  \vcenter{\hbox{\dolarger{4}$\m@th#1\mkern-2mu\times\mkern-2mu$}}} 
\newcommand*\bigtimes{\mathop{\mathpalette\@@bigtimes\relax}\displaylimits} 
\makeatother

\def\iid{\overset{\textnormal{iid}}{\sim}} % i.i.d. symbol
\def\N{\mathbb{N}}\def\R{\mathbb{R}}\def\1{\mathbbm{1}}

\def\Ccal{\mathcal{C}}\def\Hcal{\mathcal{H}}\def\Ncal{\mathcal{N}}\def\Rcal{\mathcal{R}}\def\Tcal{\mathcal{T}}\def\Ucal{\mathcal{U}}\def\Wcal{\mathcal{W}}\def\Zcal{\mathcal{Z}}

\def\Enorm{\textnormal{E}}
\def\vol{\textnormal{vol}}

\title{\bf Increasing domain asymptotics for covariate-based nonparametric Bayesian intensity estimation with Gaussian and Besov-Laplace priors}
\author{Patric Dolmeta\\
    ESOMAS Department, University of Turin\\
    and \\
    Matteo Giordano\thanks{
    M.G. has been partially supported by MUR, PRIN project 2022CLTYP4. The authors gratefully acknowledge the support from the ``de Castro" Statistics Initiative.}\\
    ESOMAS Department, University of Turin}
\date{} % set the desired data within {}, leave empty to avoid displaying the date

\begin{document}

\maketitle

\begin{abstract}
We study the problem of estimating the intensity function of a covariate-driven point process based on observations of the points and covariates over a large window. We consider the nonparametric Bayesian approach, and show that a wide class of Gaussian priors, combined with flexible link functions, achieves minimax-optimal posterior contraction rates in the increasing domain asymptotics and under the assumption that the covariates be ergodic. We also employ Besov-Laplace priors, which are popular in imaging and inverse problems due to their edge-preserving and sparsity-promoting properties. We prove that these yield optimal estimation of spatially inhomogeneous intensities belonging to Besov spaces with low integrability index. These results are based on a general concentration theorem that extends recent findings from the literature. To corroborate the theory, we provide extensive numerical simulations, implementing the considered procedures via suitable posterior sampling schemes. Further, we present two real data analyses motivated by applications in forestry and the environmental sciences. 
\end{abstract}

\vspace{24pt}


%
%\clearpage

\tableofcontents

\section{Introduction} \label{sec:Intro}

Quantifying the influence of covariates on the spatial distribution of events is a key problem in many applications with point pattern data. Consider observing the realisation of a point process $N$ and of a multivariate random field $Z$ over a Euclidean domain $\Wcal$, under the assumption that the intensity of $N$ depend on the covariates, specifically $\Enorm[N(A)] = \int_A \rho(Z(x))dx$ for all $A\subseteq \Wcal$ and some unknown function $\rho$. The goal is then to estimate $\rho$ from data $(N,Z)$. Concrete instances of this problem arise, for example, in forestry \cite{G08}, mineralogy \cite{BCST12} and the environmental sciences \cite{BGMMM20}. See also Section~\ref{Sec:RealData}.

The parametric approach to covariate-based intensity estimation has been extensively developed by both the frequentist and Bayesian communities, e.g., \cite{B78,D90,W07,RMC09, YL11}. This includes the celebrated log-Gaussian Cox model \cite{MSW98}. The available nonparametric frequentist approaches are predominantly of kernel type \cite{BRT16}, with the first asymptotic consistency results having been provided by \cite{G08} in a prototypical spatial statistics setting with increasing domains and ergodic covariates. In-fill asymptotic results for fixed observation windows were later developed by \cite{BCST12} and \cite{BGMMM20}.

On the other hand, the nonparametric Bayesian approach to the problem has been studied only more recently, after landmark methodological (e.g., ~\cite{L82,KS07,AMM09}) and theoretical (e.g., ~\cite{BSvZ15,KvZ15,DRRS17}) developments in models without covariates over the last two decades. Minimax-optimal posterior contraction rates for procedures based on Gaussian, mixture of Gaussians and P\'olya tree priors were derived by \cite{GKR25} in a similar increasing domain framework to \cite{G08}. Lastly, the preprint \cite{DG25arXiv} considered multi-bandwidth Gaussian process methods for the recovery of anisotropic intensities, proving asymptotic concentration in measurement schemes with independent and identically distributed (i.i.d.) observations of the points and covariates. The latter reference also first explored the implementation of covariate-based nonparametric Bayesian intensity estimation, devising a suitable posterior sampling algorithm and demonstrating its practical feasibility. See \cite[Section 1]{DG25arXiv} for an in-depth literature review, where many more references can be found.

%
%
%

%%%%%%%%%%%%%%%%%%%%%%%%%%%%%%%%%%%%%%%%%%%%%%%%%%%%%%%%%%%%%%%%%%%%%%%%%
\subsection{Main contributions}

This article aims to contribute to the aforementioned recent line of research. Following the nonparametric Bayesian paradigm--see \cite{GvdV17} for an overview--we model $\rho$ via prior distributions on function spaces and base our inferences on the associated posteriors, that is, the resulting conditional laws of $\rho|(N,Z)$, which provide point estimates and uncertainty quantification. Working, as in \cite{G08, GKR25}, in the increasing domain asymptotics, for priors of interest, we study the behaviour of the posterior in the infinitely informative data limit whereby $\vol(\Wcal)\to\infty$. Our first main result, Theorem \ref{Theo:GPRates}, shows that procedures based on a large class of Gaussian priors, combined with link functions satisfying a mild regularity condition, achieve, when properly tuned, optimal posterior contraction rates in $L^1$-distance. This significantly expands the scope of the previous findings of \cite{GKR25} for truncated Gaussian wavelet series priors and regular links, providing guarantees for methodologically-attractive choices such as log- and sigmoid-Gaussian priors based on standard families of covariance kernels; see \cite{MSW98,AMM09}, respectively.

We next consider Besov-Laplace priors. These belong to the broader class of Besov priors \cite{LSS09}, and enjoy significant popularity in inverse problems and imaging due to their `edge-preserving and sparsity promoting' properties that make them suitable for reconstructing `spatially inhomogeneous' functions. See e.g., ~\cite{DHS12,KLNS12,HB15} and references therein. In the context of covariate-based intensity estimation, their use--which, to our knowledge, had not been previously explored--is motivated by the realistic expectation that the intensity may present localised sharp variations due to the existence of `critical' covariate values that induce rapid changes in the probability with which the events occur. The frequentist analysis of posterior distributions arising from Besov priors has been developed in several recent articles, where optimal posterior contraction rates towards ground truths in Besov spaces containing spatially inhomogeneous functions have been established in the white noise model, density estimation and nonlinear inverse problems, among others; see \cite{ADH21,G23,AW24}. In Theorem \ref{Theo:LaplRates}, we show that correctly-tuned Besov-Laplace priors perform optimally also in the problem at hand.

The proofs are based on a refinement of the general program for posterior concentration set forth in \cite{GKR25}, where we employ concentration inequalities for ergodic fields to relate certain covariate-dependent metrics to the standard $L^1$-distance; see \ref{Sec:Proofs}. We then use rescaling techniques from the statistical theory of inverse problems, e.g., ~\cite{GN20,GR22,N23}, to obtain the required uniformity over the support of the considered priors.

To empirically illustrate our results, we devise efficient posterior sampling schemes, suited to the present increasing domain setting, employing ad-hoc dimension robust Markov chain Monte Carlo (MCMC) techniques for Gaussian and Besov priors. We then perform extensive numerical simulation studies with both spatially homogeneous and inhomogeneous ground truths, in one- and two-dimensional experimental settings. See Section \ref{Sec:Simulations}. Lastly, we employ the considered procedures in two real data analyses, with datasets from forestry and the environmental sciences, cf.~Section \ref{Sec:RealData}. A concluding discussion can be found in Section \ref{Sec:Discussion}. Further background material, all the proofs and additional empirical results are collected in the Supplement \cite{DG25Suppl}.

\section{Covariate-based nonparametric Bayesian intensity estimation}
\label{Sec:StatProbl}

For $d,D\in\N$, let $Z:=(Z(x),\ x\in\R^D)$ be a jointly measurable `covariate' random field with values in some measurable subset $\Zcal\subseteq\R^d$ (the `covariate space'). For an increasing sequence of compact `observation windows' $\Wcal_n\subseteq \Wcal_{n+1}\subset \R^D$ satisfying $\vol(\Wcal_n)\to\infty$ as $n\to\infty$,
let $Z^{(n)}:=(Z(x),\ x\in\Wcal_n)$ denote the covariates on $\Wcal_n$, and let $N^{(n)}:=\{X_1,\dots,X_{K^{(n)}}\}$ be a random point process on $\Wcal_n$ arising, conditionally given $Z^{(n)}$, as
\begin{equation}
	\label{Eq:PointProc}
		K^{(n)} |Z^{(n)} \sim \text{Po}\left(\int_{\Wcal_n}
		\lambda_\rho^{(n)}(x)dx\right); \hspace{2em} X_1,\dots,X_{K^{(n)}}|Z^{(n)},K^{(n)}
		 \iid \frac{\lambda_\rho^{(n)}(x)dx}{\int_{\Wcal_n}
			\lambda_\rho^{(n)}(x)dx},
\end{equation}
with $\lambda^{(n)}_\rho(x) := \rho(Z(x))$, $x\in\Wcal_n$, for some unknown bounded function $\rho : \Zcal\to[0,\infty)$. That is, $N^{(n)}|Z^{(n)}$ is distributed as an inhomogeneous Poisson process with first-order intensity $\lambda_\rho^{(n)}$. Unconditionally, $N^{(n)}$ is a Cox (`doubly stochastic') point process driven by the random measure $\lambda^{(n)}_\rho(x)dx$, cf.~\cite{C55}.

For some $n\in\N$, we assume that we observe a single realisation of $(N^{(n)},$ $Z^{(n)})$. The goal is then to nonparametrically estimate the unknown covariate-based intensity function $\rho$. This problem has been considered in various articles within the frequentist literature--see \cite{G08,BCST12,BGMMM20} and references therein--and recently also under the Bayesian approach by \cite{GKR25}. Following the setup in the latter, we maintain that $Z^{(n)}$ is almost surely bounded and that it defines a Borel random element in $L^\infty(\Wcal_n;\R^d)$, the Lebesgue space of $\R^d$-valued essentially bounded functions defined on $\Wcal_n$. We write $P_{Z^{(n)}}$ for the law of $Z^{(n)}$, $P^{(n)}_\rho$ for the joint distribution of the data pair $(N^{(n)},Z^{(n)})$ from \eqref{Eq:PointProc} with $Z^{(n)} \sim P_{Z^{(n)}}$, and $\Enorm^{(n)}_\rho$ for the expectation with respect to $P^{(n)}_\rho$. From the standard theory of inhomogeneous Poisson processes, it follows that $P^{(n)}_\rho$ is absolutely continuous with respect to the law $P^{(n)}_1$ arising in the homogeneous case, e.g., ~\cite[Theorem 1.3]{K98}, with likelihood
\begin{equation}
	\label{Eq:Likelihood}
	L^{(n)}(\rho)
	%:=\frac{dP^{(n)}_\rho}{dP^{(n)}_1}(N^{(n)},Z^{(n)})
	=
	e^{\sum_{k=1}^{K^{(n)}} \ln(\rho(Z(X_k)))  
		-\int_{\Wcal_n}\rho(Z(x))-1dx}.
\end{equation}

Modelling $\rho$ via a prior distribution $\Pi$ supported on the collection $\Rcal$ of bounded and non-negative-valued functions on $\Zcal$, we obtain by Bayes' formula, e.g., ~\cite[p.7]{GvdV17}, that the posterior of $\rho|(N^{(n)},Z^{(n)})$ is given by
\begin{equation}
	\label{Eq:Post}
	\Pi(R|N^{(n)},Z^{(n)})
	=\frac{\int_R L^{(n)}(\rho)d\Pi(\rho)}
	{\int_\Rcal L^{(n)}(\rho)d\Pi(\rho)},
	\qquad R\subseteq\Rcal\ \text{measurable},
\end{equation}
with $L^{(n)}$ the likelihood from \eqref{Eq:Likelihood}. For priors of interest, we study, in Section \ref{Sec:MainRes} below, the asymptotic behaviour of the posterior distributions in the `increasing domain asymptotics' $n\to\infty$. Posterior inference is illustrated via numerical simulations and real data analyses in Sections \ref{Sec:Simulations} and \ref{Sec:RealData}, respectively.

%

%\begin{remark}[Increasing domain asymptotics]%%%%%%%%%%%%%%%%%%%%%%%%%%%%
%\label{Rem:IncreasingDomains}
%\st{The scenario where a single realisation of $(N^{(n)},Z^{(n)})$ is observed is prototypical in spatial statistics, where in many applications the points and covariates have already been `sampled' at the time of data collection. See e.g., ~the case study from Section} \ref{Sec:RealData}. \st{The increasing domain asymptotics then reflects the practice of enlarging the observation window to gather more information, and is widely adopted in the literature, e.g., }~\cite{J81,J93,Z05}. \st{Other common asymptotics are of `in-fill' type, prescribing an increasing number of points within a fixed domain. The analysis of these requires different assumptions and techniques, and we refer to the recent preprint} \cite{DG25arXiv} \st{for a nonparametric Bayesian analysis of repeated observations of covariate-driven Poisson processes.}
%\end{remark}%%%%%%%%%%%%%%%%%%%%%%%%%%%%%%%%%%%%%%%%%%%%%%%%%%%%%%%%%%%%%

%
%
%
%
%

%%%%%%%%%%%%%%%%%%%%%%%%%%%%%%%%%%%%%%%%%%%%%%%%%%%%%%%%%%%%%%%%%%%%%%%%%
\section{Main results}\label{Sec:MainRes}

In this section, we present our main theoretical results quantifying the speed of posterior concentration under the `frequentist' assumption that the data $(N^{(n)},Z^{(n)})\sim P^{(n)}_{\rho_0}$ have been generated by some fixed ground truth $\rho_0\in\Rcal$. We employ the usual notion of `posterior contraction rates', namely vanishing sequences of positive numbers $\varepsilon_n\to 0$ such that, as $n\to\infty$,
$$
\Enorm_{\rho_0}^{(n)}\left[\Pi\left(\rho : \Delta(\rho,\rho_0)> \varepsilon_n|N^{(n)},Z^{(n)}\right)
\right]\to 0,
$$
where $\Delta$ is a (semi-)metric on $\Rcal$; see \cite[Chapter 8]{GvdV17} for an overview.

%
%
%

%%%%%%%%%%%%%%%%%%%%%%%%%%%%%%%%%%%%%%%%%%%%%%%%%%%%%%%%%%%%%%%%%%%%%%%%%
\subsection{Main notation and function spaces}
\label{Subsec:Notation}

We summarise notation and basic definitions for the main function spaces appearing in our asymptotic analysis. Throughout, we primarily consider covariates with values in $[0,1]^d$, cf.~Condition \ref{Cond:ErgCov}. We write $W^{\alpha,p}([0,1]^d)$, with $1\le p\le\infty$ and $\alpha\in\N$, for the Sobolev space of real-valued functions on $[0,1]^d$ with $p$-integrable (weak) derivatives up to order $\alpha$, and $\|\cdot\|_{W^{\alpha,p}}$ for its norm, cf.~\cite[Chapter 5.2]{E10}. Let $C([0,1]^d)$ be the space of continuous functions, equipped with the supremum norm $\|\cdot\|_{L^\infty}$, and, for $\alpha>0$, let $C^\alpha([0,1]^d)$ be the H\"older space of $\lfloor \alpha \rfloor$-times differentiable functions with $(\alpha - \lfloor \alpha \rfloor)$-H\"older continuous $\lfloor \alpha \rfloor^{\textnormal{th}}$ derivative, with norm $\|\cdot\|_{C^\alpha}$, cf.~\cite[Chapter 5.1]{E10}.

We next define wavelet-based Besov spaces. Let $(\psi_\ell, \ \ell\in\N)$ be a single-index reordering of an orthonormal wavelet basis of $L^2([0,1]^d)$ comprising  $S$-regular, $S\in\N$, compactly supported and boundary corrected Daubechies wavelets, e.g., ~\cite[Appendix A]{LSS09}. For $0\le \alpha< S$ and $1\le p<\infty$, let
$$
B^\alpha_p([0,1]^d)
:= \left\{f=\sum_{\ell=1}^\infty f_\ell \psi_\ell : \| f \|^p_{B^\alpha_p}
:=\sum_{\ell=1}^\infty \ell^{p(\alpha/d+1/2)-1}|f_\ell|^p<\infty \right\}.
$$
As the underlying wavelet basis can be constructed with arbitrarily high regularity, we shall implicitly assume that the condition $\alpha\le S$ be verified throughout. By the wavelet characterisation of Hilbert-Sobolev spaces, e.g., ~\cite[p.~370]{GN16}, for $\alpha\in\N$, it holds that $B^\alpha_2([0,1]^d) = W^{\alpha,2}([0,1]^d)$ with norm equivalence. For $\alpha>0$, we then use the special notation $H^\alpha([0,1]^d) = B^\alpha_2([0,1]^d)$, $\|\cdot\|_{H^\alpha} = \|\cdot\|_{B^\alpha_2}$. When $p<2$, the spaces $B^\alpha_p([0,1]^d)$ are known to contain `spatially inhomogeneous' functions, e.g., ~ones that are predominantly flat but may present localised sharp variations (or even jumps), cf.~\cite[Section 1]{DJ98}. Among these, the most relevant scale is for  $p=1$; for instance, $B^1_1([0,1]^d)$ is closely related to the space of bounded variation functions, e.g., ~\cite[Section 1]{DJ98}, which is widely used in imaging and signal processing.

%
%
%
%
%

%%%%%%%%%%%%%%%%%%%%%%%%%%%%%%%%%%%%%%%%%%%%%%%%%%%%%%%%%%%%%%%%%%%%%%%%%
\subsection{Assumptions on the covariates }\label{Subsec:Covariates}

We adopt two key working conditions on the covariates to drive the required accumulation of information in the increasing domain asymptotics, assuming $Z$ to be stationary and ergodic, cf.~Remark \ref{Rem:ErgCov}.

\begin{condition}\label{Cond:ErgCov}%%%%%%%%%%%%%%%%%%%%%%%%%%%%%%%%%%%%%
	Let the random field $Z$ be stationary, and denote by $\nu_Z$ its stationary distribution (i.e.,~$Z(x)\sim \nu_Z$ for all $x\in\R^D$). Assume that $\nu_Z$ be supported on $\Zcal = [0,1]^d$, and be absolutely continuous with bounded probability density function (p.d.f.). Further, assume that for all $K_1>0$ there exists some $K_2>0$ such that, for any $f\in W^{1,\infty}([0,1]^d)$ satisfying $\|f\|_{W^{1,\infty}}\le K_1$, the following concentration inequality holds for all $t>0$ and $n\in\N$,
	\begin{equation}
		\label{Eq:ConcIneq}
		P_{Z^{(n)}}\left(\left| \frac{1}{\vol(\Wcal_n)}\int_{\Wcal_n} f(Z(x))dx - \int_{[0,1]^d} f(z)d\nu_Z(z) \right| > t \right) \\
		\le e^{-K_2\vol(\Wcal_n)t^2}.
	\end{equation}
\end{condition}%%%%%%%%%%%%%%%%%%%%%%%%%%%%%%%%%%%%%%%%%%%%%%%%%%%%%%%%%%

Stationarity is commonly assumed in spatial statistics, e.g., ~\cite[Chapter 2.3]{C15}, and here implies the (testable, \cite{BS17}) scenario where the distributional properties of the covariates do not vary across the domain. The requirement that $\nu_Z$ be supported on $[0,1]^d$ is merely for concreteness and implies no loss of generality since, if $Z$ took values outside this set, we could transform it through a smooth and invertible map $\Phi:\R^d\to[0,1]^d$, proceed in the analysis with such `pre-processed' covariates, and then convert the results back to the original scale via $\Phi^{-1}$. See Example \ref{Ex:GaussCov} in the Appendix \cite{DG25Suppl} for a concrete construction of $\Phi$. The concentration inequality \eqref{Eq:ConcIneq} is a quantitative version of the `ergodic property' whereby spatial averages of $Z$ over the increasing domain $\Wcal_n$ converge to their expectation. Similar to the recent investigation of \cite{GKR25}, this is used in the proofs to control the covariate-based metric \eqref{Eq:EmpDist} below, which is central to the contraction rate analysis. 

In the Appendix \ref{App:AddMaterial} we present two major classes of random fields satisfying Condition \ref{Cond:ErgCov}, based on stationary and ergodic Gaussian processes and Poisson random tessellations, respectively.

\begin{remark}[Ergodic covariates]%%%%%%%%%%%%%%%%%%%%%%%%
	\label{Rem:ErgCov}
	In the increasing domain asymptotics, the ergodicity assumption enables $Z^{(n)}$ to fully explore the covariate space $[0,1]^d$, and implies that similar covariate values are recorded multiple times across the observation window, which is often realistic in applications, e.g., ~\cite{R00,DVJ03,C15}. Through stationarity, the information carried by distinct sub-regions with similar covariate values can then be combined, allowing consistent inferences. Analogous assumptions underpin the kernel consistency results of \cite{G08}, as well as the posterior contraction rates for various nonparametric Bayesian procedures recently obtained by \cite{GKR25}. 
\end{remark}%%%%%%%%%%%%%%%%%%%%%%%%%%%%%%%%%%%%%%%%%%%%%%%%%%%%%%%%%%%%%

Under Condition \ref{Cond:ErgCov}, the amount of information contained in the data is determined by $\vol(\Wcal_n)\to\infty$. For convenience and notational simplicity, we take $\vol(\Wcal_n) = n$ and work with square observation windows
\begin{align}
	\label{Eq:SpatialWn}
	\Wcal_n=\left[-\frac{1}{2}n^{1/D},\frac{1}{2}n^{1/D}\right]^D.
\end{align}
The posterior contraction rates derived below (expressed in $n$) should then be thought of as being in terms of $\vol(\Wcal_n)$. Extensions to general domains containing increasing `regularly-shaped' sets such as \eqref{Eq:SpatialWn} are straightforward.

%%%%%%%%%%%%%%%%%%%%%%%%%%%%%%%%%%%%%%%%%%%%%%%%%%%%%%%%%%%%%%%%%%%%%%%%%%
\subsection{Posterior contraction rates for Gaussian priors}
\label{Subsec:GPRates}

Gaussian priors are extremely popular across statistics and machine learning \cite{RW06}. They have been successfully employed in non-covariate-based nonparametric Bayesian intensity estimation, e.g., ~\cite{MSW98,AMM09,KvZ15}, and recently also in models with covariates by \cite{GKR25}, in the form of truncated Gaussian wavelet series priors combined with `regular' link functions.

Here, we consider a significantly wider class of Gaussian priors and link functions. Specifically, under Condition \ref{Cond:ErgCov}, we model the unknown covariate-based intensity $\rho:[0,1]^d\to[0,\infty)$ via $n$-dependent priors $\Pi_n$, constructed starting from the following base Gaussian measures.

\begin{condition}\label{Cond:GPCondition}%%%%%%%%%%%%%%%%%%%%%%%%%%%%%%%%
	Let $\Pi$ be a centred Gaussian Borel probability
	measure on $C^1([0,1]^d)$, and assume that, for some $\alpha>d/2$, the reproducing kernel Hilbert space (RKHS) $\Hcal$ of $\Pi$ is equal to $H^{\alpha}([0,1]^d)$, with RKHS norm $\|\cdot\|_\Hcal\simeq \|\cdot\|_{H^\alpha}$. 
\end{condition}%%%%%%%%%%%%%%%%%%%%%%%%%%%%%%%%%%%%%%%%%%%%%%%%%%%%%%%%%%

If $K(x,y)$, $x,y\in [0,1]^d$, is the prior covariance kernel, $\Hcal$ equals the completion of the space of functions $f = \sum_{i=1}^k \alpha_i K(s_i,\cdot)$, $k\in\N$, $\alpha_1,\ldots,\alpha_k \in \R$, $s_1, \ldots ,s_k\in [0,1]^d $, relative to the norm induced by the inner product 
$$
\left\langle \sum_{i=1}^k a_i K(s_i, \cdot), \sum_{j=1}^\ell b_j 
K(t_j, \cdot) \right\rangle_\Hcal 
= \sum_{i=1}^k \sum_{j=1}^\ell a_i b_j K(s_i,t_j).
$$
This characterises the prior information geometry, regulating, under Condition \ref{Cond:GPCondition}, the smoothness (i.e., the degree of differentiability) of the draws from $\Pi$ via the hyper-parameter $\alpha$. See \cite[Chapter 11]{GvdV17} for definitions and an overview on the use of Gaussian processes in Bayesian nonparametrics.

Condition \ref{Cond:GPCondition} includes Gaussian series expansions on orthonormal bases spanning the Sobolev scale (e.g., ~the Fourier and wavelet bases, cf.~\cite[Section 11.4.5]{GvdV17}), as well as the laws of stationary Gaussian processes with algebraically decaying spectral measures (e.g., ~the popular Matérn processes, cf.~\cite[Section 11.4.4]{GvdV17}). See Appendix \ref{App:GPs} for further details.

Given $\Pi$ as in Condition \ref{Cond:GPCondition}, let $W:=(W(z),\ z\in[0,1]^d)\sim \Pi$ and let $\Pi_n$ be the law of the following random function,
\begin{equation}
	\label{Eq:RescaledGP}
	R_n(z) := \eta\left[n^{-\frac{d}{4\alpha+2d}} W(z)\right], \qquad z\in[0,1]^d,
\end{equation}
where $\eta:\R\to[0,\infty)$ is any smooth and strictly increasing link function. The introduction of $\eta$ guarantees that the realisations of $R_n$ take non-negative values. The mild requirement that it be smooth and strictly increasing allows for ample flexibility in its choice; for example, exponential (as in \cite{MSW98}), sigmoid-type (as in \cite{AMM09}) or soft-max-type (e.g., ~as in \cite[Section 3.2.1]{GKR25}) links can be used. The $n$-dependent rescaling of $W$ in \eqref{Eq:RescaledGP} is a technique borrowed from the statistical theory of inverse problems \cite{GN20,N23, AW24}. In the proofs, it implies a high-probability bound for the $C^1$-norm of $R_n$ (cf.~\ref{Subsec:ProofGPRates}), which crucially enables a uniform application of the concentration inequality \eqref{Eq:ConcIneq}. Non-asymptotically, such rescaling involves a mere shrinkage of the covariance kernel of $\Pi$, and does not require further tuning in practice.

\begin{theorem}\label{Theo:GPRates}%%%%%%%%%%%%%%%%%%%%%%%%%%%%%%%%%%%%%
	Let $\rho_0 = \eta\circ w_0$ for some $w_0\in C^\beta([0,1]^d)$ and some $\beta>\max(1,d/2)$. Consider data $(N^{(n)},Z^{(n)})\sim P^{(n)}_{\rho_0}$ from \eqref{Eq:PointProc} with $\rho=\rho_0$ and with $Z$ satisfying Condition \ref{Cond:ErgCov}. Let the rescaled Gaussian prior $\Pi_n$ be given by the law of the random function \eqref{Eq:RescaledGP}, with $\Pi$ satisfying Condition \ref{Cond:GPCondition} for some $d/2<\alpha\le\beta$. Then, for $M>0$ large enough, as $n\to\infty$,
	$$
	\Enorm_{\rho_0}^{(n)}
	\Bigg[\Pi_n\Big(\rho : \|\rho - \rho_0\|_{L^1([0,1]^d,\nu_Z)} 
	> M n^{-\frac{\alpha}{2\alpha+d}}
	\Big| N^{(n)},Z^{(n)}\Big)\Bigg]
	\to 0.
	$$
\end{theorem}%%%%%%%%%%%%%%%%%%%%%%%%%%%%%%%%%%%%%%%%%%%%%%%%%%%%%%%%%%%%%

This shows that the posterior concentrates over neighbourhoods of the ground truth whose $L^1([0,1]^d,\nu_Z)$-radius shrinks at rate $n^{-\alpha/(2\alpha+d)}$. The hyper-parameter $\alpha$ from Condition \ref{Cond:GPCondition} determines the regularity of the elements of the RKHS $\Hcal$ of the base Gaussian measure $\Pi$; the condition that $\alpha\le \beta$ then amounts to the requirement that $\Hcal$ not `over-smooth' $\rho_0$. Clearly, Theorem \ref{Theo:GPRates} yields the fastest rate when $\alpha = \beta$, namely when the regularity of $\Hcal$ exactly matches the true smoothness. This is in line with the aforementioned inverse problems literature \cite{N23}. In particular, under the concrete examples of covariate random fields from Appendix \ref{App:AddMaterial}, smoothness-matching priors concentrate in standard $L^1$-distance at rate $n^{-\beta/(2\beta+1)}$, which is known to be minimax-optimal for $\beta$-smooth intensities, e.g., ~\cite[Chapter 6.2]{K98}.

The proof of Theorem \ref{Theo:GPRates} is based on a refinement of recent results by \cite{GKR25}, which establish a general program to derive posterior contraction rates in the increasing domain asymptotics; see Theorem \ref{Theo:GenTheo}. In fact, the argument shows that $\Pi(\cdot|N^{(n)},Z^{(n)})$ achieves the same rate also in the `empirical' $L^1$-distance
\begin{equation}
	\label{Eq:EmpDist}
	\frac{1}{n}\|\lambda_\rho^{(n)} - \lambda_{\rho_0}^{(n)}\|_{L^1(\Wcal_n)}
	=\frac{1}{|\Wcal_n|}\int_{\Wcal_n}|\rho(Z(x))-\rho_0(Z(x))|dx,
\end{equation}
which is related to the $L^1([0,1]^d,\nu_Z)$-metric via the concentration inequality \eqref{Eq:ConcIneq}. This implies that the considered rescaled Gaussian priors also solve the `prediction' problem of estimating the spatial intensity function $\lambda_\rho^{(n)}$ in \eqref{Eq:PointProc}.

%
%
%

%%%%%%%%%%%%%%%%%%%%%%%%%%%%%%%%%%%%%%%%%%%%%%%%%%%%%%%%%%%%%%%%%%%%%%%%%
\subsection{Posterior contraction rates for Besov-Laplace priors}
\label{Subsec:LaplRates}

The second class of priors we consider are the Besov-Laplace ones, which are popular in imaging and inverse problems due to their capability of reconstructing spatially inhomogeneous functions \cite{LSS09,DHS12,KLNS12,HB15,ADH21}. In the present setting, they represent a natural prior model for intensities that exhibit sharp variations corresponding to specific `critical' covariate values; see Section \ref{Sec:RealData} and Appendix \ref{Suppl:RealData} for illustrations with real data.

Under Condition \ref{Cond:ErgCov}, we construct rescaled Besov-Laplace priors for the function $\rho:[0,1]^d\to[0,\infty)$ as follows. Let $(\psi_\ell, \ \ell\in\N)$ be a (single-index reordering of an) orthonormal wavelet basis of $L^2([0,1]^d)$, cf.~Section \ref{Subsec:Notation} for details. For any $\alpha>d$, let $\Pi$ be a random wavelet series prior given by
\begin{equation}
	\label{Eq:BaseLapalPrior}
	W\sim \Pi,
	\quad
	W(z) = \sum_{\ell=1}^\infty \ell^{-(\frac{\alpha}{d}-\frac{1}{2})} w_\ell \psi_\ell(z), 
	\quad z\in[0,1]^d,
	\quad w_\ell\iid \text{Laplace},
\end{equation}
where the Laplace (or double exponential) distribution on $\R$ has p.d.f.~proportional to $e^{-|t|/2}$, $t\in\R$. In the language of \cite{LSS09}, $\Pi$ is called a `$B^\alpha_{11}$-prior', while \cite{ADH21} terms it an `$(\alpha-d)$-regular Laplace prior' since  $\Pi(B^\gamma_{p}) = 1$ for all $\gamma<\alpha-d$ and all $p\in[1,\infty]$, cf.~\cite[Lemma 5.2]{ADH21}. Similar to Section \ref{Subsec:GPRates}, we transform $W\sim \Pi$ from \eqref{Eq:BaseLapalPrior} via a suitable $n$-dependent rescaling and the application of a smooth and strictly increasing link function $\eta:\R\to[0,\infty)$,
\begin{equation}
	\label{Eq:RescaledLaplPrior}
	R_n(z) 
	:= \eta\left[n^{-\frac{d}{2\alpha+d}} W(z)\right], \qquad z\in[0,1]^d,
\end{equation}
denoting by $\Pi_n$ the law of $R_n$. Note that the rescaling in  \eqref{Eq:RescaledLaplPrior} is of different order compared to the one  for Gaussian priors in \eqref{Eq:RescaledGP}. In accordance with the existing asymptotic theory of Besov priors \cite{ADH21}, this turns out to be the correct shrinkage to achieve optimal posterior contraction rates over Besov spaces containing spatially inhomogeneous functions, as the next result shows.

\begin{theorem}\label{Theo:LaplRates}%%%%%%%%%%%%%%%%%%%%%%%%%%%%%%%%%%%%%
	Let $\rho_0 = \eta\circ w_0$ for some $w_0\in B_1^\beta([0,1]^d)$ and some $\beta>1+d$. Consider data $(N^{(n)},Z^{(n)})\sim P^{(n)}_{\rho_0}$ from \eqref{Eq:PointProc} with $\rho=\rho_0$ and with $Z$ satisfying Condition \ref{Cond:ErgCov}. Let the rescaled Besov-Laplace prior $\Pi_n$ be given by the law of the random function \eqref{Eq:RescaledLaplPrior}, with $W$ as in \eqref{Eq:BaseLapalPrior} for some $1+d<\alpha\le\beta$. Then, for $M>0$ large enough, as $n\to\infty$,
	$$
	\Enorm_{\rho_0}^{(n)}
	\Bigg[\Pi_n\Big(\rho : \|\rho - \rho_0\|_{L^1([0,1]^d,\nu_Z)} 
	> M n^{-\frac{\alpha}{2\alpha+d}}
	\Big| N^{(n)},Z^{(n)}\Big)\Bigg]
	\to 0.
	$$
\end{theorem}%%%%%%%%%%%%%%%%%%%%%%%%%%%%%%%%%%%%%%%%%%%%%%%%%%%%%%%%%%%%%

Theorem \ref{Theo:LaplRates} states that the posterior concentrates around (possibly) spatially inhomogeneous ground truths $\rho_0\in B^\beta_1([0,1]^d)$ at rate $n^{-\alpha/(2\alpha+d)}$. When $\alpha = \beta$, i.e.,~when the prior regularity matches the true smoothness, this retrieves the minimax-optimal rate, e.g., ~\cite[Chapter 10]{HKPT12}. In contrast, procedures based on Gaussian priors (in fact, all linear methods) have been shown to perform sub-optimally in the presence of spatial inhomogeneities even in the simpler white noise model \cite[Theorem 4.1]{AW24} (see also \cite[Theorem 1]{DJ98}). Therefore, under the assumptions of Theorem \ref{Theo:LaplRates}, they cannot generally be expected to achieve the minimax rates, regardless of their tuning.

%
%
%
%
%

%%%%%%%%%%%%%%%%%%%%%%%%%%%%%%%%%%%%%%%%%%%%%%%%%%%%%%%%%%%%%%%%%%%%%%%%%%
\section{Simulation studies}\label{Sec:Simulations}

We conducted several simulation studies to test the considered methods in practice. On increasing bidimensional square domains as in \eqref{Eq:SpatialWn}, we considered scenarios with one and two covariates, both with spatially homogeneous and inhomogeneous ground truths. Here, we start by presenting the results for the univariate case, deferring the rest to Appendix \ref{Suppl:AddSimul}.

%
%
%

%%%%%%%%%%%%%%%%%%%%%%%%%%%%%%%%%%%%%%%%%%%%%%%%%%%%%%%%%%%%%%%%%%%%%%%%%%
\subsection{Experiments with univariate covariates; homogeneous ground truth}
\label{Subsec:1DHom}

We take $Z$ as a (centred and unit variance) square-exponential Gaussian process with length scale $0.5$, e.g., ~\cite[p.~84]{RW06}, transformed via the standard normal c.d.f.~in accordance with Example \ref{Ex:GaussCov}. We set the ground truth to
\begin{equation}
	\label{Eq:1DTruth}
	\rho_0(z) = 100 f_{SN}(z; 0.8, 0.3 ,-5),
	\qquad z\in[0,1],
\end{equation}
with $f_{SN}$ the skew-normal p.d.f., see Figure \ref{Fig:1DEstim} below. The resulting point patterns concentrate in the regions with covariate values near $0.65$, with around $100$ expected points per unit area; see Figure \ref{Fig:1DObs}.

\begin{figure}[H]%%%%%%%%%%%%%%%%%%%%%%%%%%%%%%%%%%%%%%%%%%%%%%%%%%%%%%%%
	\centering
	\includegraphics[width=\linewidth]{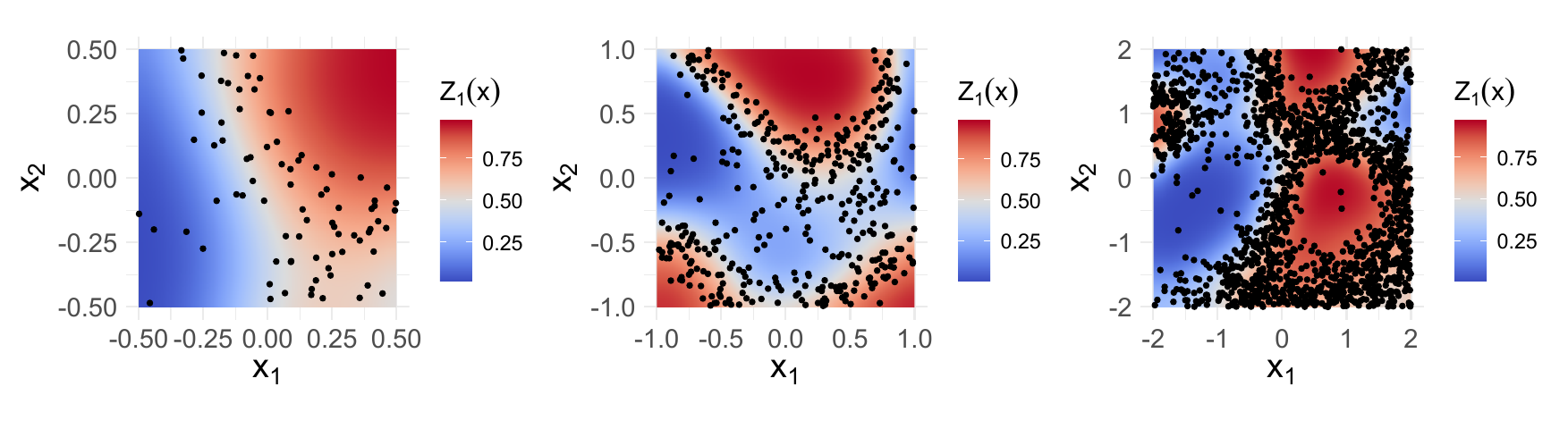}
	\caption{Three realizations of the centred and unit variance square-exponential Gaussian covariate process with length scale $0.5$, mapped to $[0,1]$ via the standard normal c.d.f.. Panels refer to an increasing domain scenario, where $\Wcal_n = [-\sqrt n/2,\sqrt n/2]^2$ for $n = \{ 1,4,16\}$, respectively. Black points represent realizations of the point process with ground truth \eqref{Eq:1DTruth}. Their expected number is $100$ times each window's volume and they concentrate in regions with covariate values near $0.65$.}
	\label{Fig:1DObs}
\end{figure}%%%%%%%%%%%%%%%%%%%%%%%%%%%%%%%%%%%%%%%%%%%%%%%%%%%%%%%%%%%%%%

For each $n$, we sample $Z^{(n)}$ by discretising it over a pre-specified grid on $\Wcal_n$, as described for example in \cite[p.~670]{MOLL07}. Given the covariates, the points are drawn using the implementation in the $\texttt{R}$ package $\texttt{spatstat}$ \cite{BRT16} of the `thinning' procedure from \cite[Section 2.3]{AMM09}. We then perform posterior inference with Gaussian wavelet series (e.g.,~Example \ref{Ex:GPWav} in the Appendix) and Besov-Laplace priors (cf.~Section \ref{Subsec:LaplRates}), combined with the exponential link, employing ad-hoc dimension-robust MCMC algorithms to draw approximate samples from the posterior distributions; see Appendix \ref{App:Algorithms} for details.

In Figure \ref{Fig:1DEstim}, we plot the obtained (MCMC approximations to the) posterior means $\hat\rho_\Pi^{(n)} :=\Enorm^\Pi[\rho|N^{(n)},Z^{(n)}]$, for $n=4, 16, 256$, alongside the associated point-wise $95\%$-credible intervals. For comparison, we also include nonparametric kernel estimates $\hat\rho^{(n)}_\kappa$ produced by the built-in routine in $\texttt{spatstat}$. Consistency results for related methods in the increasing domain asymptotics were proved by \cite{G08}. Table \ref{Tab:1DEstim} reports the (numerically approximated) relative $L^1$-estimation errors, averaged across 50 replications of each experiment. In line with the theory, the empirical results display a clear convergence towards the ground truth as $n$ increases. The performances for Gaussian and Besov-Laplace priors were very close, as expected in view of the spatial homogeneity of $\rho_0$ in \eqref{Eq:1DTruth}. The kernel method also achieved similar estimation errors, with only marginally worse performances for the larger areas.

\begin{figure}[H]%%%%%%%%%%%%%%%%%%%%%%%%%%%%%%%%%%%%%%%%%%%%%%%%%%%%%%%%%
	\centering
	\includegraphics[width=\linewidth]{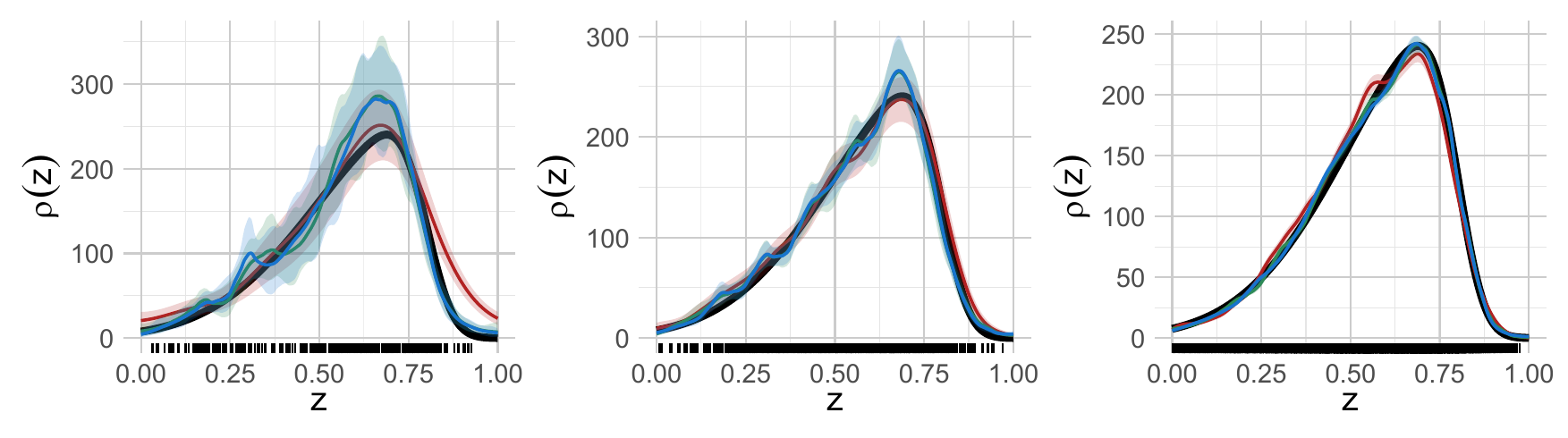}
	\caption{Left to right: Posterior means for Gaussian (solid green) and Besov-Laplace (solid blue) priors, pointwise $95\%$-credible intervals (shaded regions) and kernel estimates (solid red) for $n = \{ 4, 16, 256\}$, respectively. The ground truth $\rho_0$ from \eqref{Eq:1DTruth} is shown in solid black. Rugs at the bottom represent the covariate values at the point locations.}
	\label{Fig:1DEstim}
\end{figure}%%%%%%%%%%%%%%%%%%%%%%%%%%%%%%%%%%%%%%%%%%%%%%%%%%%%%%%%%%%%%%

\begin{table}[H]%%%%%%%%%%%%%%%%%%%%%%%%%%%%%%%%%%%%%%%%%%%%%%%%%%%%%%%%%%
	\caption{Relative $L^1$-estimation errors 
	$\|\hat\rho - \rho_0\|_{L^1}/\|\rho_0\|_{L^1}$ 
	for posterior mean estimators obtained from Gaussian and Besov--Laplace priors, and for the kernel estimator $\hat\rho_\kappa^{(n)}$, averaged over $50$ repeated experiments; standard deviations are reported in parentheses. The Gaussian and Besov--Laplace priors employ the oracle-tuned regularity parameters $\alpha=1.5$ and $\alpha=1$, respectively, selected via grid-search over $\alpha\in[0.05,5]$. For reference, $\|\rho_0\|_{L^1}=99.23$.}
	\centering
	\hrule
	\begin{tabular}{cccc}
		\makecell{\\
			$n$} & 
		\makecell{Gaussian ($\alpha = 0.5$)\\
			$\frac{\|\hat\rho_\Pi^{(n)} - \rho_0\|_{L^1}}{\|\rho_0\|_{L^1}}$}
		& 
		\makecell{Laplace ($\alpha = 0.5$)\\
			$\frac{\|\hat\rho_\Pi^{(n)} - \rho_0\|_{L^1}}{\|\rho_0\|_{L^1}}$}
		& \makecell{\\
			$\frac{\|\hat\rho_\kappa^{(n)} - \rho_0\|_{L^1}}{\|\rho_0\|_{L^1}}$} \\
		1
		& 0.25 (0.06) & 
		0.32 (0.04) & 
		0.21 (0.06) \\    
		%\hline
		4
		& 0.18 (0.06) & 
		0.16 (0.03) &
		0.15 (0.05) \\    
		%\hline
		16 &  
		0.08 (0.03) &  
		0.09 (0.02) & 
		0.10 (0.03) \\
		64 &  
		0.04 (0.01) &  
		0.03 (0.01) & 
		0.06 (0.01) \\
		256 &  
		0.03 (0.005) &  
		0.02 (0.005) & 
		0.06 (0.01) \\
	\end{tabular}
	\hrule
	\label{Tab:1DEstim}
\end{table}%%%%%%%%%%%%%%%%%%%%%%%%%%%%%%%%%%%%%%%%%%%%%%%%%%%%%%%%%%%%%%%

Concerning the wavelet basis, we used the Daubechies-8 maximally symmetric 
(i.e.,\\`Symmelet-8') functions, with symmetric boundary reflection, implemented in the $\texttt{R}$ package $\texttt{wavethresh}$. The series were truncated after the first $1,024$ terms, sufficient to ensure that the approximation errors are negligible compared to the statistical ones in all considered experimental settings. With these specifications, MCMC running times ranged from 5 to 18 minutes on an Intel(R) Core(TM) i7-10875H 2.30GHz processor with 32 GB of RAM, with longer executions for the experiments over larger domains. Each MCMC run was  initialised at cold (constant) starts and iterated for 25,000 steps,  with 10,000 discarded initial samples as the burn-in. The step-sizes for the MCMC algorithms were tuned in order to obtain a stabilisation of the acceptance probabilities between 20\% and 30\% after burn-in.

To select the prior regularity parameters in \eqref{Eq:GPWav} and \eqref{Eq:BaseLapalPrior}, we  tested a wide range of equally-spaced choices $\alpha\in[0.05,5]$; the results in Table \ref{Tab:1DEstim} refer to the values yielding the best reconstructions (which is reported at the top). We note that this is an oracle tuning choice unfeasible in practice, here employed for the purpose of illustration. Furthermore, while our theoretical results require the prior regularity to be greater than a certain minimal value (respectively, $\alpha>d/2$ for Gaussian priors and $\alpha>1+d$ for Besov-Laplace priors), the grid-search was extended below these thresholds.  The `optimal' regularity $\alpha =1$ selected for the Besov-Laplace priors hints at the applicability of Theorem \ref{Theo:LaplRates} even in low-regularity scenarios, even though such extension is technically challenging. We also acknowledge that the slight discrepancy between the effect of the smoothness in the theoretical results and the simulation studies may be due to finite-sample effects and discretisation errors within the employed posterior sampling algorithms.

For real-world scenarios, where the above oracle tuning is unavailable, our findings suggest to adopt a conservative choice of $\alpha$ close to the lower thresholds reported in Theorems \ref{Theo:GPRates} and \ref{Theo:LaplRates}, which guarantees consistency (albeit at the cost of a possibly sub-optimal contraction rate). An attractive alternative is to select $\alpha$ in a data-driven way, either by randomising it within a hierarchical Bayesian model, e.g., ~\cite[Chapter 11.6]{GvdV17}, or by replacing it with a preliminary estimate $\hat \alpha$ following the empirical Bayes approach, e.g., ~\cite{Rousseau2017asymptotic}. We provide a brief exploration of the former strategy in Section \ref{Subsec:Adaptive} below, where we consider hierarchical Gaussian and Besov-Laplace priors obtained by assigning to $\alpha$ a standard non-informative exponential hyper-prior.

%
%
%

%%%%%%%%%%%%%%%%%%%%%%%%%%%%%%%%%%%%%%%%%%%%%%%%%%%%%%%%%%%%%%%%%%%%%%%%%%
\subsection{Experiments with univariate covariates; inhomogeneous ground truth}\label{Subsec:1DInhom}

We next consider a spatially inhomogeneous ground truth $\rho_0$ modelled after the `Blocks' function introduced in \cite{DJ94}, suitably modified and rescaled to verify the positivity constraint and to enforce an expected number of points per unit area equal to 100. Specifically,
\begin{equation}
	\rho_0(z) = \frac{100}{1.64} \left(1.01 +\sum_{m=1}^6 h_m \frac{1+ \sgn(z - z_m)}{2} \right),
	\qquad z\in[0,1],
	\label{Eq:1D_block}
\end{equation}
with $h\in\{3,-4,3.1,-2.2,3.1,-3\}$ the vector of block amplitudes and $z\in\{0.1,0.15,0.25,0.40,0.71,0.81\}$ the block locations, cf., Figure \ref{Fig:1d_block}.

With this setup, we proceed performing analogous numerical experiments to the ones described in Section \ref{Subsec:1DHom}, implementing posterior inference with Gaussian wavelet series and Besov-Laplace priors, truncating the expansions at level $1, 024$, and selecting the prior regularity parameters after testing the values from an equally-spaced grid $\alpha \in [0.05, 5]$.

Figure \ref{Fig:1d_block} displays the obtained posterior means and pointwise $95\%$ credible intervals for increasing sample sizes $n=\{4,16,256\}$. Kernel estimates are also included. The relative $L^1$-estimation errors, averaged over 50 replications of each experiment, are reported in Table \ref{Tab:1D_block}. Here, kernel estimates evidently oversmooth the ground truth, and their associated estimation errors only marginally improve for larger $n$. From a comparison with the excellent reconstructions of the spatially homogeneous ground truth from Section \ref{Subsec:1DHom}, this hints at the known sub-optimality of linear procedures in the presence of spatial inhomogeneities \cite{DJ98}, and also at the possibly unfavourable behaviour in this setting of the automatic bandwidth selection mechanism implemented in \texttt{spatstat}, which follows Silverman’s rule-of-thumb \cite{S86}. In contrast, the performance of the Gaussian and Besov-Laplace priors is clearly superior, with a progressively more precise detection of the blocky-structure of $\rho_0$ that takes advantage of the underlying wavelet structure and its localisation properties. Between these, the Besov-Laplace priors achieves lower estimation errors across all sample sizes. This is in line with the minimax-optimal posterior contraction rates for Besov-regular ground truths established in Theorem \ref{Theo:LaplRates}, which Gaussian priors are expected to be unable to match \cite{AW24}.

\begin{figure}[H]%%%%%%%%%%%%%%%%%%%%%%%%%%%%%%%%%%%%%%%%%%%%%%%%%%%%%%%%%%
	\centering
	\includegraphics[width=\linewidth]{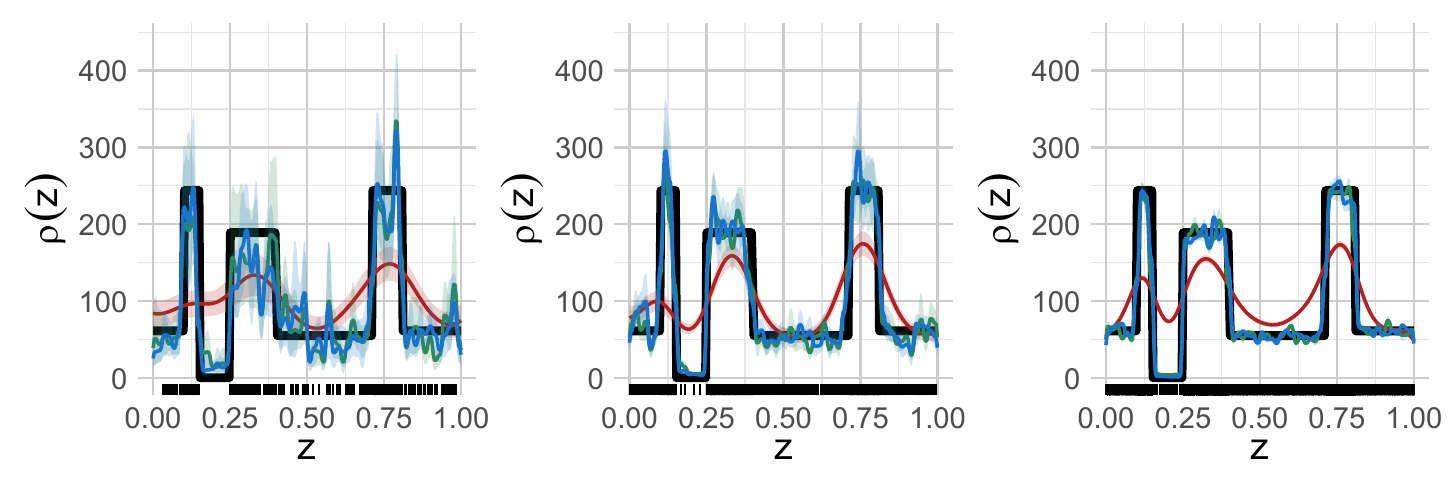}
	\caption{Left to right: Posterior means for Gaussian (solid green) and Laplace (solid blue) priors, pointwise $95\%$-credible intervals (shaded regions) and kernel estimates (solid red) for $n = \{ 4, 16, 256\}$, respectively. The ground truth $\rho_0$ from \eqref{Eq:1D_block} is shown in solid black.}
	\label{Fig:1d_block}
\end{figure}%%%%%%%%%%%%%%%%%%%%%%%%%%%%%%%%%%%%%%%%%%%%%%%%%%%%%%%%%%%%%%%

\begin{table}[H]%%%%%%%%%%%%%%%%%%%%%%%%%%%%%%%%%%%%%%%%%%%%%%%%%%%%%%%%%%%
\caption{Relative $L^1$-estimation errors 
	$\|\hat\rho-\rho_0\|_{L^1}/\|\rho_0\|_{L^1}$ 
	for posterior mean estimators obtained from Gaussian and Besov--Laplace priors, and for the kernel estimator $\hat\rho_\kappa^{(n)}$, averaged over $50$ repeated experiments; standard deviations are reported in parentheses. Both the Gaussian and Besov--Laplace priors employ the oracle-tuned regularity parameter $\alpha=0.5$, selected via grid-search over $\alpha\in[0.05,5]$. For reference, $\|\rho_0\|_{L^1}=100.96$.}
	\hrule
	\centering
	\begin{tabular}{cccc}
		\makecell{\\
			$n$} & 
		\makecell{Gaussian ($\alpha = 0.5$)\\
			$\frac{\|\hat\rho_\Pi^{(n)} - \rho_0\|_{L^1}}{\|\rho_0\|_{L^1}}$}
		& 
		\makecell{Laplace ($\alpha = 0.5$)\\
			$\frac{\|\hat\rho_\Pi^{(n)} - \rho_0\|_{L^1}}{\|\rho_0\|_{L^1}}$}
		& \makecell{\\
		$\frac{\|\hat\rho_\kappa^{(n)} - \rho_0\|_{L^1}}{\|\rho_0\|_{L^1}}$} \\

		1   & 0.67 (0.04) & 0.64 (0.03) & 0.59 (0.02) \\
		4   & 0.51 (0.03) & 0.48 (0.02) & 0.54 (0.02) \\
		16  & 0.33 (0.03) & 0.30 (0.04) & 0.48 (0.02) \\
		64  & 0.20 (0.01) & 0.15 (0.01) & 0.43 (0.01) \\
		256 & 0.12 (0.01) & 0.10 (0.02) & 0.45 (0.01) \\
	\end{tabular}
	\hrule
	\label{Tab:1D_block}
\end{table}

%
%
%
%

%%%%%%%%%%%%%%%%%%%%%%%%%%%%%%%%%%%%%%%%%%%%%%%%%%%%%%%%%%%%%%%%%%%%%%%%%%
\subsection{Hierarchical Gaussian and Besov-Laplace priors}
\label{Subsec:Adaptive}

As remarked after the statements of Theorems \ref{Theo:GPRates} and \ref{Theo:LaplRates}, and in the discussion of the empirical results in Section \ref{Subsec:1DHom}, the performance of the considered Gaussian and Besov-Laplace priors is heavily influenced by the choice of the regularity hyper-parameter $\alpha$. In the Bayesian approach, a principle strategy to tackle the selection of hyper-parameters is via hierarchical prior models, where the hyper-parameters are themselves treated as being random and assigned suitable hyper-priors, e.g., ~\cite[Chapter 11.6]{GvdV17}.

Here, we report the empirical performances, under the same experimental settings of Sections \ref{Subsec:1DHom} and \ref{Subsec:1DInhom} (and Appendix \ref{Suppl:AddSimul}), of hierarchical extensions of the Gaussian wavelet series (cf.~Example \ref{Ex:GPWav} in the Appendix) and the Besov-Laplace priors (see Section \ref{Subsec:LaplRates}) considered therein, with a standard exponential hyper-prior $\alpha\sim\text{Exp(1)}$. We note that our theory provides no theoretical guarantees for the resulting posterior distributions, whose asymptotic behaviours in the present setting with covariate-driven point processes observed over increasing domains present interesting open questions for future research; see Section \ref{Sec:Discussion} for further discussion. The purpose of this section is then to provide preliminary numerical evidence of the feasibility and reliability of such data-driven tuning approach for the problem at hand.

For the ground truths \eqref{Eq:1DTruth} and \eqref{Eq:1D_block} (and \eqref{Eq:2D_skn} and \eqref{Eq:2D_bs} from Appendix \ref{Suppl:AddSimul}), Table \ref{Tab:adapt} below displays the obtained relative $L^1$-estimation errors of the posterior means resulting from the hierarchical Gaussian and Besov-Laplace priors, for $n = 256$. The posterior mean of the regularity hyper-parameter $\alpha$, and 95\% credible intervals for it, are also included. The results are averaged over $50$ repetitions of each experiment, and standard deviations for the estimation errors are reported. Across the experiments, the hierarchical procedures obtained extremely close estimation errors to the ones scored by the fixed-regularity methods with oracle tuning of Sections \ref{Subsec:1DHom}, \ref{Subsec:1DInhom} and Appendix \ref{Suppl:AddSimul} (cf.~Tables \ref{Tab:1DEstim}, \ref{Tab:1D_block}, \ref{Tab:2D_skn} and \ref{Tab:2D_bs}, last rows). This hints at the capability of the hierarchical Gaussian and Besov-Laplace priors to flexibly adapt to the unknown characteristics of the true covariate-based intensity function, both under spatial homogeneity and inhomogeneity. The hierarchical procedure thus represents a natural and attractive choice for real-world scenarios where no information about the regularity of the inferential target is available.

In Table \ref{Tab:adapt}, we again note similar performances of both procedures for the spatially homogeneous ground truths \eqref{Eq:1DTruth} and \eqref{Eq:2D_skn}, while the hierarchical Besov-Laplace prior achieved lower estimation errors under the spatially inhomogeneous models \eqref{Eq:1D_block} and \eqref{Eq:2D_bs}. The obtained $95\%$ credible intervals for the regularity hyper-parameter $\alpha$ are compatible with the values from the oracle tuning via grid-search employed in the non-hierarchical case. The recorded slight discrepancies between the posterior means $\hat \alpha$ and the oracle hyper-parameters (e.g., ~for the Besov-Laplace prior and the spatially inhomogeneous ground truth \eqref{Eq:1D_block}) appear to be ultimately uninfluential in terms of the overall estimation error of the covariate-based intensity function. 

%\begin{table}[H]%%%%%%%%%%%%%%%%%%%%%%%%%%%%%%%%%%%%%%%%%%%%%%%%%%%%%%%%%%%
%\centering
%\begin{tabular}{c|c|c|c|c|}
%& \multicolumn{2}{c|}{Hierarchical Gaussian} & \multicolumn{2}{c|}{Hierarchical Laplace} \\
%$\rho_0$ & $\frac{\|\hat\rho_\Pi^{(n)} - \rho_0\|_{L^1}}{\|\rho_0\|_{L^1}}$ & $\hat\alpha$ & $\frac{\|\hat\rho_\Pi^{(n)} - \rho_0\|_{L^1}}{\|\rho_0\|_{L^1}}$ & $\hat\alpha $ \\
%\hline
%\eqref{Eq:1DTruth}  & 0.03 (0.007) & 1.43 (0.39)   & 0.02 (0.003) & 1.04 (0.21) \\
%\eqref{Eq:1D_block}   & 0.12 (0.01) &  0.76 ( 0.08) & 0.10 (0.02) & 1.08 (0.14) \\
%\eqref{Eq:2D_skn}  & 0.11 (0.009) & 2.02 (0.10) & 0.098 (0.01) & 2.06 (0.11) \\
%\eqref{Eq:2D_bs}  & 0.47 (0.05) &  0.96 (0.05) & 0.39 (0.01) &  1.44 (0.03) \\
%\hline
%\end{tabular}
%\caption{Averaged relative $L^1$-estimation errors (and their standard deviation) over 50 repeated experiments for the ground truths \eqref{Eq:1DTruth} and \eqref{Eq:1D_block}, and averaged posterior mean and $95\%$-credible intervals of the smoothness parameter. }
%    \label{Tab:adapt}
%\end{table}%%%%%%%%%%%%%%%%%%%%%%%%%%%%%%%%%%%%%%%%%%%%%%%%%%%%%%%%%%%%%%%%

\begin{table}[H]%%%%%%%%%%%%%%%%%%%%%%%%%%%%%%%%%%%%%%%%%%%%%%%%%%%%%%%%%%%
\caption{Averaged relative $L^1$-estimation errors 
	$\|\hat\rho-\rho_0\|_{L^1}/\|\rho_0\|_{L^1}$ (with standard deviations in parentheses) over $50$ repeated experiments for the ground truths \eqref{Eq:1DTruth}, \eqref{Eq:1D_block}, \eqref{Eq:2D_skn}, and \eqref{Eq:2D_bs}. The table reports results for hierarchical Gaussian and hierarchical Besov--Laplace priors, where the regularity parameter $\alpha$ is assigned a hyperprior and inferred jointly with the intensity function. For each setting, we also report the posterior mean of $\alpha$ together with pointwise $95\%$ credible intervals. The results illustrate the adaptive behaviour of the hierarchical approach, which automatically selects different levels of smoothness depending on the underlying regularity of the true intensity.}
	\hrule
	\centering
	\begin{tabular}{ccccc}
		& \multicolumn{2}{c}{Hierarchical Gaussian} & \multicolumn{2}{c}{Hierarchical Laplace} \\
		\hline
		$\rho_0$ & $\frac{\|\hat\rho_\Pi^{(n)} - \rho_0\|_{L^1}}{\|\rho_0\|_{L^1}}$ & $\hat\alpha \ [95\%\text{-CI}] $ & $\frac{\|\hat\rho_\Pi^{(n)} - \rho_0\|_{L^1}}{\|\rho_0\|_{L^1}}$ & $\hat\alpha \ [95\%\text{-CI}] $ \\

		\eqref{Eq:1DTruth}  & 0.03 (0.007) & 1.43 [0.95, 2.33]   & 0.02 (0.003) & 1.04 [0.62, 1.48] \\
		\eqref{Eq:1D_block}   & 0.12 (0.01) &  0.76 [0.21,1.31] & 0.10 (0.02) & 1.08 [0.50,1.58] \\
		\eqref{Eq:2D_skn}  & 0.11 (0.009) & 2.02 [1.42,2.60] & 0.098 (0.01) & 2.25 [1.54,2.86] \\
		\eqref{Eq:2D_bs}  & 0.47 (0.05) &  0.96 [0.46,1.49] & 0.39 (0.01) &  1.44 [0.89,1.98] \\
		\hline
	\end{tabular}
	\hrule
	\label{Tab:adapt}
\end{table}%%%%%%%%%%%%%%%%%%%%%%%%%%%%%%%%%%%%%%%%%%%%%%%%%%%%%%%%%%%%%%%%

%\begin{table}[H]
%\centering
%\begin{tabular}{c|c|c|}
%Model & \textbf{Gaussian} & \textbf{Laplace} \\
%\hline
%\eqref{Eq:1DTruth}  & 1.43 (1.25, 2.03) & 1.04 (0.92, 1.18) \\
%\eqref{Eq:1D_block}   & 0.76 (0.71, 0.81) &  1.04 (1.00, 1.08) \\
%\eqref{Eq:2D_skn}  & 2.02 (1.92, 2.10) &  2.25 (2.14, 2.36) \\
%\eqref{Eq:2D_bs}  &  0.96 (0.91, 0.99) &  1.44 (1.39, 1.48) \\
%\hline
%\end{tabular}
%\caption{Average posterior mean and $95\%$ credible intervals of the smoothness parameter $\alpha$ over 50 repeated experiments. }
%    \label{Tab:adapt}
%\end{table}

Computationally, the above hierarchical procedures require more sophisticated MCMC algorithms for posterior sampling, where the base routines to implement the fixed-regularity methods of Sections \ref{Subsec:1DHom} and \ref{Subsec:1DInhom} are embedded in a Metropolis-within-Gibbs structure that alternates approximate draws from the conditional posterior distributions of the covariate-based intensity and the regularity hyper-parameters. We provide details in Appendix \ref{App:AdaptiveMCMC}. Per-experiment running times ranged between 45 and 90 minutes.

%
%
%
%

%%%%%%%%%%%%%%%%%%%%%%%%%%%%%%%%%%%%%%%%%%%%%%%%%%%%%%%%%%%%%%%%%%%%%%%%%%
\section{Applications to a tropical rainforest dataset}
\label{Sec:RealData}

Covariate-driven point processes naturally arise in the analysis of forestry data, where studying the influence of location-dependent environmental variables on the spatial distribution of trees is often of primary interest. In this section, we present an application to a tropical rainforest dataset, containing the position of 3,605 \textit{Beilschmiedia pendula Lauraceae} (BEI) trees in a rectangular area of $1000 \times 500$ metres  in Barro Colorado Island \cite{BRT16}, with detailed information on the terrain elevation and steepness. See Figure \ref{Fig:BEI}. Similar data were investigated, among others, by \cite{G08}, using nonparametric kernel methods, and by \cite{MOLL07}, within a parametric Bayesian model based on the log-Gaussian Cox process. These studies broadly highlight the preference of this species of trees for medium-to-high terrain and steep slopes.

\begin{figure}[H]%%%%%%%%%%%%%%%%%%%%%%%%%%%%%%%%%%%%%%%%%%%%%%%%%%%%%%%%%
	\centering
	\includegraphics[width=\linewidth]{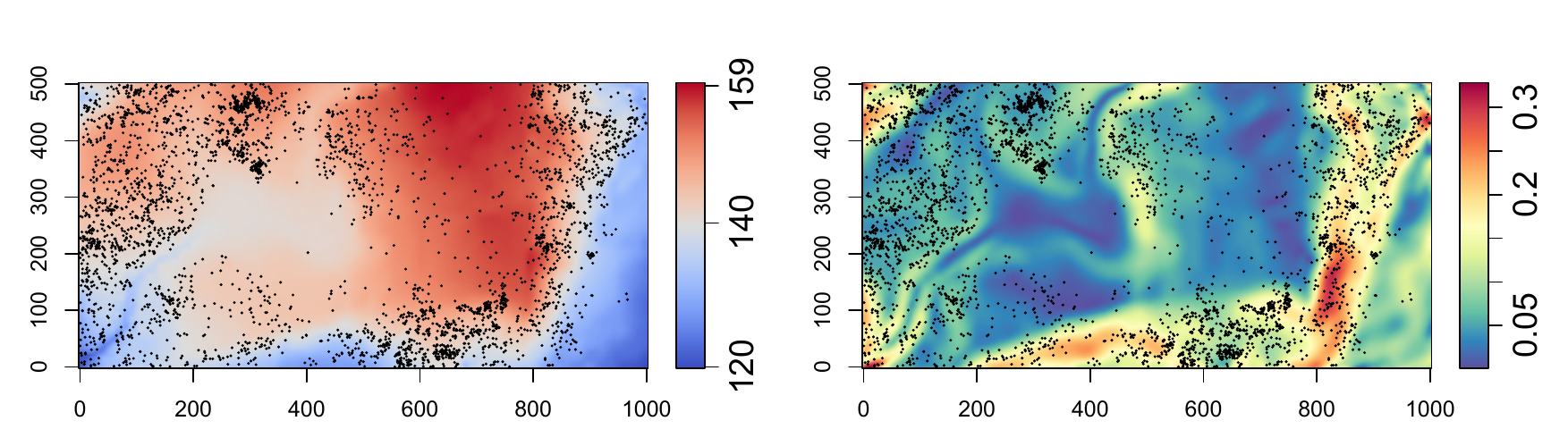}
	\caption{Terrain elevation (left, in metres) and terrain steepness (right, measured as the norm of the altitude gradient) for the BEI rainforest dataset from Barro Colorado Island. The black points represent the observed locations of the $3605$ \textit{Beilschmiedia pendula Lauraceae} trees within the $1000\times500$ metre observation window.}
	\label{Fig:BEI}
\end{figure}%%%%%%%%%%%%%%%%%%%%%%%%%%%%%%%%%%%%%%%%%%%%%%%%%%%%%%%%%%%%%%

A distinctive feature of the BEI dataset is that the trees appear to be mostly distributed along specific contour lines, with sharp variations in their density corresponding to changes in elevation and steepness. Motivated by this, we base our analysis on the Besov-Laplace priors from Section \ref{Subsec:LaplRates}, with the goal of achieving a clear separation between high- and low-intensity zones.

We start with an exploratory step in which we model the intensity as a function of each covariate individually. Figure \ref{Fig:BEI1DEstim} shows the obtained posterior means and point-wise credible sets, computed via the MCMC algorithm described in Appendix \ref{App:Algorithms}. As benchmark, the plot also includes the kernel estimates from \cite{BRT16}, alongside the associated pointwise 95\%-confidence intervals. The results are in mutual agreement and fall in line with the previous findings from the literature. The estimated elevation-based intensities steadily increase up to global maxima located near 150 metres, followed by sharp drops. For the slope, sharp raises are detected at gradient norms around 0.04, after which the steepness-based intensities stabilise. The sharp increase is particularly evident in the results for the Besov-Laplace prior, stemming from its edge-preserving properties. Within the identified trends, the estimates exhibit internal variability and detect local modes. As observed in \cite{MOLL07}, these fluctuations are likely spurious and may be caused by small scale clustering, between-species competition and/or unobserved covariates.

\begin{figure}[H]%%%%%%%%%%%%%%%%%%%%%%%%%%%%%%%%%%%%%%%%%%%%%%%%%%%%%%%%%
	\centering
	\includegraphics[width=\linewidth]{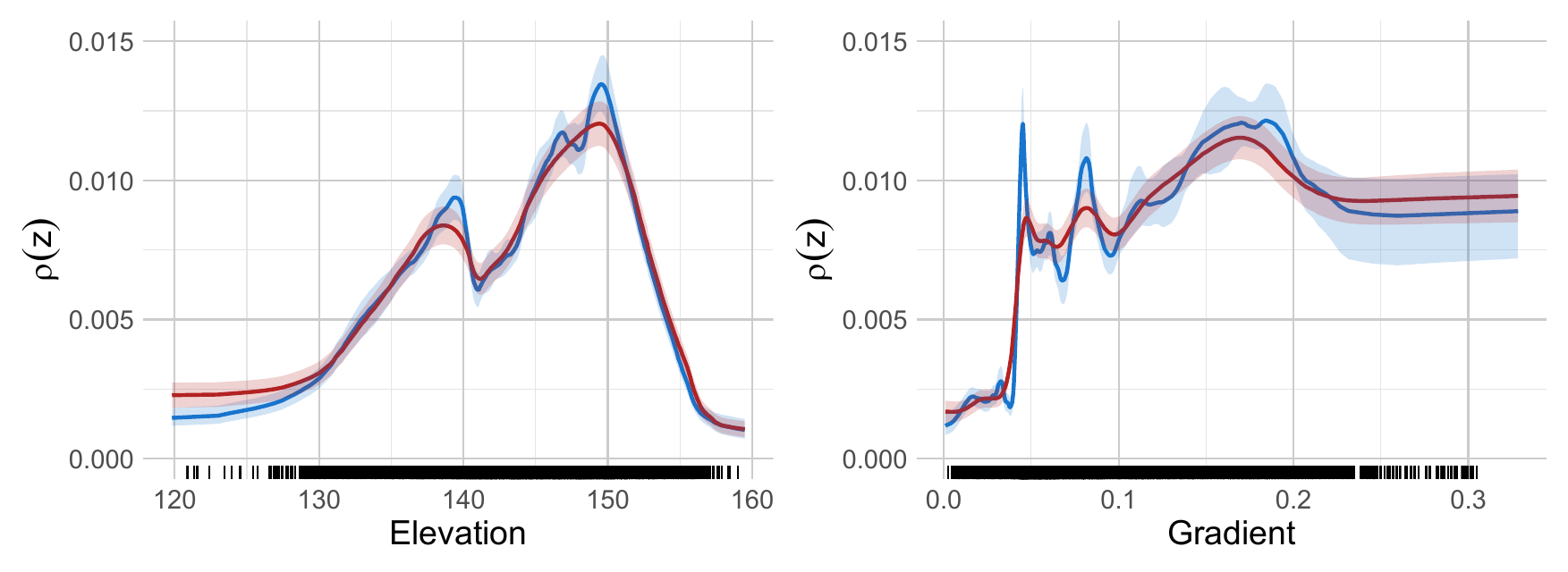}
	\caption{Posterior estimates of the covariate-dependent intensity function based on terrain elevation (left) and terrain steepness (right) in the BEI dataset. The solid blue curves represent the posterior means under the Besov--Laplace prior with regularity parameter $\alpha=1$, while the shaded blue regions denote the associated pointwise $95\%$-credible intervals. The red curves correspond to the kernel-based estimators from \cite{BRT16}, together with their pointwise $95\%$-confidence intervals. Both methods identify preferred elevation ranges around $150$ metres and sharp increases in intensity for slope values near $0.04$, with the Besov--Laplace prior producing sharper transitions due to its edge-preserving properties.}
	\label{Fig:BEI1DEstim}
\end{figure}%%%%%%%%%%%%%%%%%%%%%%%%%%%%%%%%%%%%%%%%%%%%%%%%%%%%%%%%%%%%%%

Turning to the full analysis, Figure \ref{Fig:BEI2DEstim} shows the posterior mean and the kernel estimate jointly based on elevation and slope. These broadly reinforce the observations from the preceding univariate study, identifying a particularly favourable range around 150 meters of altitude, as well as a sharp increase in the intensity for gradient norms around 0.04, across almost all elevations. Compared to the kernel method, the Besov-Laplace posterior mean more clearly separates the covariate values with high and low intensity. For example, a steep drop is detected for altitudes exceeding 155 metres, similar to the one displayed in the left panel of Figure \ref{Fig:BEI1DEstim}. Lastly, in Figure \ref{Fig:BEI2DSpatial}, we plot the plug-in posterior mean $\hat\rho^{(n)}_\Pi\circ Z^{(n)}$ and plug-in kernel estimate $\hat\rho^{(n)}_\kappa\circ Z^{(n)}$ of the spatial intensity function $\lambda^{(n)}_\rho$ from \eqref{Eq:PointProc}. Both closely match the terrain contour lines from Figure \ref{Fig:BEI} (right), providing convincing reconstructions of the point pattern and corroborating the notion that the density of trees be mainly influenced by the steepness, e.g.,~\cite[Example 10]{MOLL07} and \cite[Section 5]{G08}. Again, the results for the Besov-Laplace prior provide a sharper separation between the zones in the observation window with high and low spatial intensity, as can be seen for example from a comparison of the regions in the bottom-right corners of the two panels of Figure \ref{Fig:BEI2DSpatial}.

\begin{figure}[H]%%%%%%%%%%%%%%%%%%%%%%%%%%%%%%%%%%%%%%%%%%%%%%%%%%%%%%%%%
	\centering
	\includegraphics[width=\linewidth]{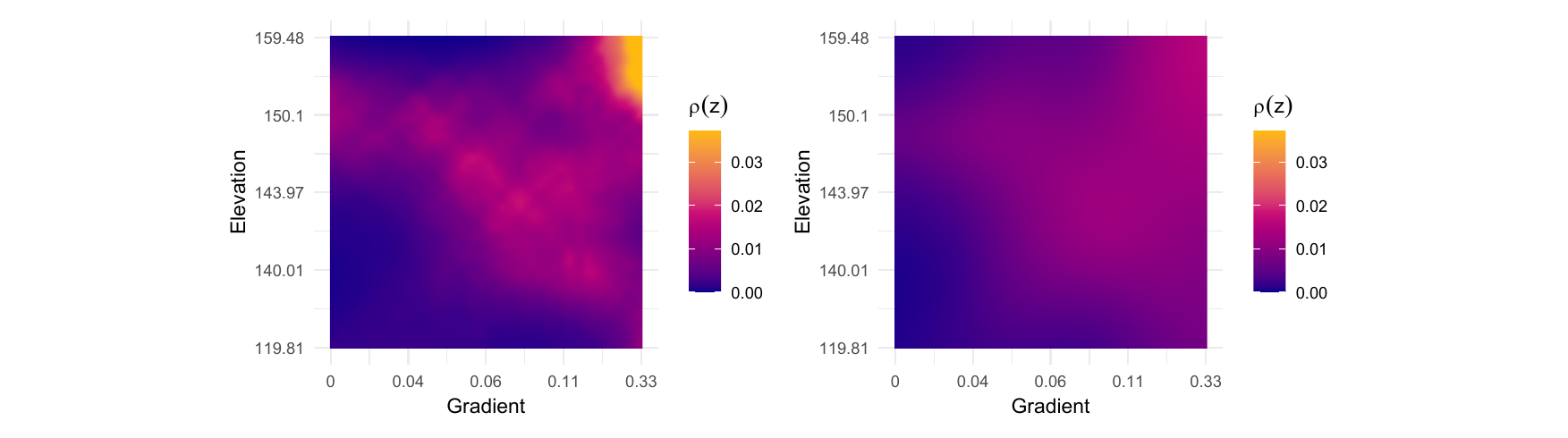}
	\caption{Posterior mean estimate of the intensity function under the Besov--Laplace prior with regularity parameter $\alpha=1$ (left), and corresponding kernel estimate (right), jointly modelled as a function of terrain elevation and terrain steepness in the BEI dataset. Both approaches identify a favourable elevation range around $150$ metres and a marked increase in intensity for gradient norms near $0.04$. Compared to the kernel estimator, the Besov--Laplace posterior mean yields a sharper separation between regions of high and low intensity, particularly for elevations exceeding approximately $155$ metres.}
	\label{Fig:BEI2DEstim}
\end{figure}%%%%%%%%%%%%%%%%%%%%%%%%%%%%%%%%%%%%%%%%%%%%%%%%%%%%%%%%%%%%%%

\begin{figure}[H]%%%%%%%%%%%%%%%%%%%%%%%%%%%%%%%%%%%%%%%%%%%%%%%%%%%%%%%%%
	\centering
	\includegraphics[width=\linewidth]{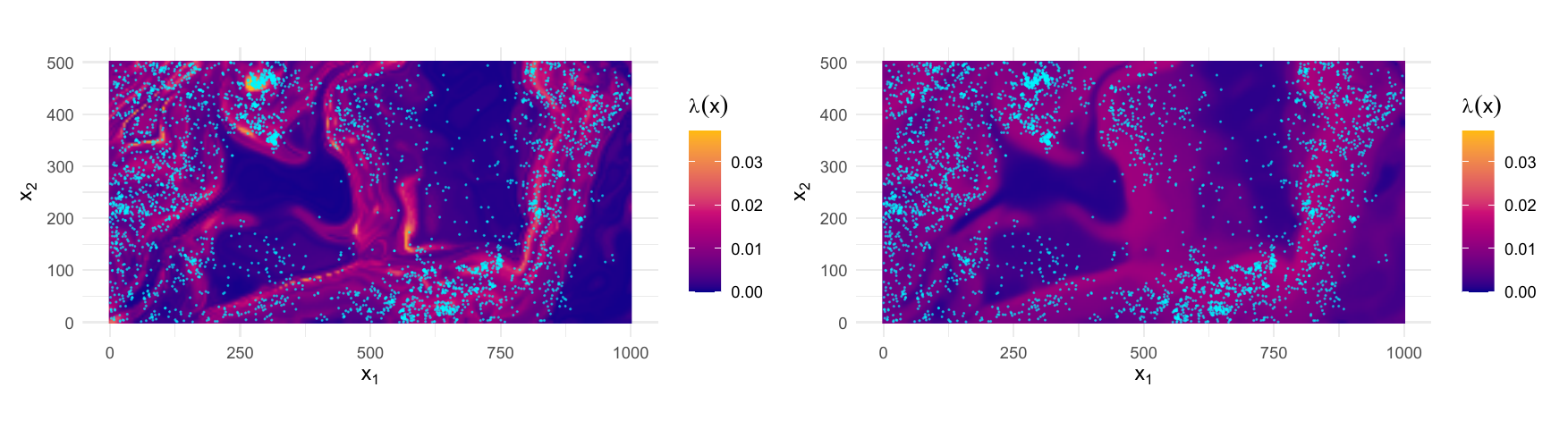}
	\caption{Plug-in estimates of the spatial intensity function $\lambda_\rho^{(n)}$ over the BEI observation window obtained from the posterior mean under the Besov--Laplace prior with regularity parameter $\alpha=1$ (left) and from the kernel estimator (right). The estimated spatial intensities closely follow the terrain contour structure, supporting the interpretation that tree density is strongly influenced by terrain steepness. Relative to the kernel estimator, the Besov--Laplace approach provides a sharper delineation between regions of high and low spatial intensity, particularly near the boundaries of steep terrain areas.}
	\label{Fig:BEI2DSpatial}
\end{figure}%%%%%%%%%%%%%%%%%%%%%%%%%%%%%%%%%%%%%%%%%%%%%%%%%%%%%%%%%%%%%%

%

%In \ref{Suppl:RealData} we present another application, analysing a Canadian wildfire dataset with meteorological covariates.

%
%
%
%
%

%%%%%%%%%%%%%%%%%%%%%%%%%%%%%%%%%%%%%%%%%%%%%%%%%%%%%%%%%%%%%%%%%%%%%%%%%%
\section{Discussion}
\label{Sec:Discussion}

%We have considered nonparametric Bayesian intensity estimation based on the observation of a covariate-driven Poisson process over a large domain. We have provided optimal theoretical guarantees for a wide class of Gaussian priors and for Besov-Laplace priors, demonstrating their feasibility and usefulness in practice via simulations and real data analyses. 
The increasing domain asymptotics considered throughout is prototypical in spatial statistics, e.g., ~\cite{J81,J93,Z05}, since in many applications points and covariates have already been ‘sampled’ at the time of data collection, and more information can be gathered only by enlarging the observation window. Other common asymptotics are of ‘in-fill’ type, with an increasing number of points within a fixed domain. The study of these require different assumptions and techniques, and we refer to the preprint \cite{DG25arXiv} for a nonparametric Bayesian analysis of i.i.d.~observations of covariate-driven Poisson processes.

A limitation of our results is that they are non-adaptive, since both Theorems \ref{Theo:GPRates} and \ref{Theo:LaplRates} require to correctly tune the prior regularity to the smoothness of the ground truth to achieve optimal rates. 
While hierarchical procedures based on Gaussian and Besov priors (similar to the ones explored empirically in Section \ref{Subsec:Adaptive}) have been shown to be able to achieve automatic adaptation in simpler statistical models, cf.~\cite{vdVvZ09,AS24}, pursuing this extension in the present setting is technically challenging, primarily because of the necessity, within our proof strategy, to employ concentration inequalities for spatial averages like \eqref{Eq:ConcIneq} uniformly over the support of the prior.

Another important open question concerns the frequentist validity of the associated uncertainty quantification. Indeed, in infinite-dimensional statistical models, it is generally known that, even for consistent posteriors, the associated credible sets may have asymptotically vanishing coverage, e.g., ~\cite{DF86}. This issue could be approached via semiparametric and nonparametric Bernstein-von Mises-type theorems, which have been established, in different statistical models, for both Gaussian, e.g., ~\cite{CN13}, and Besov-Laplace priors \cite{GR25}. %While some positive numerical results are provided in \ref{App:Coverage}, 
We leave these investigations for future work.

\section*{Data availability statement}
The \texttt{R} code to replicate the simulation study and real data analyses is available at:\\ \href{https://github.com/PatricDolmeta/covariate-based-nonparametric-Bayesian-intensity-estimation-with-Gaussian-and-Besov-Laplace-priors}{\texttt{https://github.com/PatricDolmeta}}.

\section*{Acknowledgments}

The authors are grateful to the Associate Editor and two anonymous Referees for helpful comments that lead to an improvement of the manuscript.
The authors thank the ``de Castro'' Statistics Initiative and Collegio Carlo Alberto for supporting this research. M.G.~also acknowledges the partial financial support by MUR, PRIN project 2022CLTYP4. 

\bibliography{Bibliography.bib}

@article{rousseau2017asymptotic,
  title={Asymptotic behaviour of the empirical bayes posteriors associated to maximum marginal likelihood estimator},
  author={Rousseau, Judith and Szabo, Botond and others},
  journal={ANNALS OF STATISTICS},
  volume={45},
  number={2},
  pages={833--865},
  year={2017}
}

@inproceedings{AMM09,
	author = {Adams, Ryan Prescott and Murray, Iain and MacKay, David J. C.},
	booktitle = {Proceedings of the 26th Annual International Conference on Machine Learning},
	numpages = {8},
	pages = {9--16},
	series = {ICML '09},
	title = {Tractable nonparametric {B}ayesian inference in {P}oisson processes with {G}aussian process intensities},
	year = {2009},
}

@article{ADH21,
  title={Rates of contraction of posterior distributions based on $p$-exponential priors},
  author={Agapiou, Sergios and Dashti, Masoumeh and Helin, Tapio},
  journal={{B}ernoulli},
  volume={27},
  number={3},
  pages={1616--1642},
  year={2021}
}

@article{AS24,
  title={Adaptive inference over {B}esov spaces in the white noise model using $p$-exponential priors},
  author={Agapiou, Sergios and Savva, Aimilia},
  journal={{B}ernoulli},
  volume={30},
  number={3},
  pages={2275--2300},
  year={2024}
}

@article{AW24,
  title={{L}aplace priors and spatial inhomogeneity in {B}ayesian Inverse Probl.},
  author={Agapiou, Sergios and Wang, Sven},
  journal={{B}ernoulli},
  volume={30},
  number={2},
  pages={878--910},
  year={2024}
}

@article{BCST12,
author = {Baddeley, Adrian and Chang, Ya-Mei and Song, Yong and Turner, Rolf},
year = {2012},
month = {01},
pages = {221-236},
title = {Nonparametric estimation of the dependence of a spatial point process on spatial covariates},
volume = {5},
journal = {Stat. Interface}
}

@book{BRT16,
  title={Spatial point patterns: methodology and applications with R},
  author={Baddeley, Adrian and Rubak, Ege and Turner, Rolf},
  volume={1},
  year={2016},
  publisher={CRC press Boca Raton}
}

@article {BS17,
    AUTHOR = {Bandyopadhyay, Soutir and Subba Rao, Suhasini},
    TITLE = {A test for stationarity for irregularly spaced spatial data},
   	JOURNAL = {J. Roy. Statist. Soc. Ser. B},
  	FJOURNAL = {J. Roy. Statist. Soc. Ser. B},
    VOLUME = {79},
    YEAR = {2017},
    NUMBER = {1},
    PAGES = {95--123},
    ISSN = {1369-7412,1467-9868},
   	MRCLASS = {62M30 (60G10 62E20 62M07 62M15)},
  	MRNUMBER = {3597966}
}

@article {BSvZ15,
    AUTHOR = {Belitser, Eduard and Serra, Paulo and van Zanten, Harry},
     TITLE = {Rate-optimal {B}ayesian intensity smoothing for inhomogeneous {P}oisson processes},
   	JOURNAL = {J. Statist. Plann. Inference},
  	FJOURNAL = {Journal of Statistical Planning and Inference},
    VOLUME = {166},
    YEAR = {2015},
    PAGES = {24--35},
    ISSN = {0378-3758,1873-1171},
   	MRCLASS = {62F15 (60G55 62G05)},
  	MRNUMBER = {3390131},
    DOI = {10.1016/j.jspi.2014.03.009},
}

@article{BGMMM20,
	author = {Borrajo, M. I. and Gonz\'{a}lez-Manteiga, W. and Mart\'{\i}nez-Miranda, M. D.},
	fjournal = {Computational Statistics \& Data Analysis},
	issn = {0167-9473},
	journal = {Comput. Statist. Data Anal.},
	pages = {106875, 21},
	title = {Bootstrapping kernel intensity estimation for inhomogeneous point processes with spatial covariates},
	volume = {144},
	year = {2020}}

@incollection{B78,
	author = {Brillinger, David R.},
	booktitle = {Developments in statistics, {V}ol. 1},
	isbn = {0-12-426601-0},
	mrclass = {62M10 (60G55)},
	mrnumber = {501668},
	mrreviewer = {Alan\ G.\ Hawkes},
	pages = {33--133},
	publisher = {Academic Press, New York-London},
	title = {Comparative aspects of the study of ordinary time series and of point processes},
	year = {1978}}

@article{CN13,
    author = {Isma{\"e}l Castillo and Richard Nickl},
    title = {{Nonparametric {B}ernstein–von {M}ises theorems in {G}aussian white noise}},
    volume = {41},
    journal = {Ann. Statist.},
    number = {4},
    publisher = {Institute of Mathematical Statistics},
    pages = {1999--2028},
    year = {2013}}

@article{CSRW13,
   title={{MCMC} Methods for Functions: Modifying Old Algorithms to Make Them Faster},
   volume={28},
   number={3},
   journal={Statistical Science},
   author={Cotter, S. L. and Roberts, G. O. and Stuart, A. M. and White, D.},
   year={2013}
}

@article {C55,
    AUTHOR = {Cox, D. R.},
    TITLE = {Some statistical methods connected with series of events},
   	JOURNAL = {J. Roy. Statist. Soc. Ser. B},
  	FJOURNAL = {J. Roy. Statist. Soc. Ser. B},
    VOLUME = {17},
    YEAR = {1955},
    PAGES = {129--157; discussion, 157--164},
    ISSN = {0035-9246},
   	MRCLASS = {62.0X},
  	MRNUMBER = {92301}
}

@book {C15,
    AUTHOR = {Cressie, Noel A. C.},
    TITLE = {Statistics for Spatial Data},
    SERIES = {Wiley Classics Library},
   	EDITION = {Revised},
 	PUBLISHER = {John Wiley \& Sons, Inc., New York},
    YEAR = {2015},
    PAGES = {xx+900},
    ISBN = {978-1-119-11461-1},
   	MRCLASS = {62M30 (01A75 60D05 60G55 62H11 62M40)},
  	MRNUMBER = {3559472},
}

@book{DVJ03,
  title={An introduction to the theory of point processes: volume I: elementary theory and methods},
  author={Daley, Daryl J and Vere-Jones, David},
  year={2003},
  publisher={Springer}
}

@article{DHS12,
  title={{B}esov priors for {B}ayesian Inverse Probl.},
  author={Dashti, Masoumeh and Harris, Stephen and Stuart, Andrew},
  journal={Inverse Probl. Imaging},
  volume={6},
  number={2},
  pages={183--200},
  year={2012}
}

@article{DF86,
  title={On the Consistency of {B}ayes Estimates},
  author={Diaconis, Persi and Freedman, David},
  journal={Ann. Statist.},
  volume={14},
  number={1},
  pages={1--26},
  year={1986},
  publisher={Institute of Mathematical Statistics}
}

@article{D90,
	author = {Peter J. Diggle},
	issn = {09641998, 1467985X},
	journal = {J. Roy. Statist. Soc. Ser. A},
	number = {3},
	pages = {349--362},
	publisher = {[Wiley, Royal Statistical Society]},
	title = {A point process modelling approach to raised incidence of a rare phenomenon in the vicinity of a prespecified point},
	urldate = {2023-10-10},
	volume = {153},
	year = {1990}}

@article{DG25arXiv,
  title={A nonparametric {B}ayesian analysis of independent and identically distributed observations of covariate-driven {P}oisson processes},
  author={Dolmeta, Patric and Giordano, Matteo},
  journal={arXiv preprint arXiv:2509.02299},
  year={2025}
}

@article{DG25Suppl,
  title={Supplement to ``Increasing domain asymptotics for covariate-based nonparametric {B}ayesian intensity estimation''},
  author={Dolmeta, Patric and Giordano, Matteo},
  journal={},
  year={2025}
}

@article{DJ94,
  title={Ideal spatial adaptation by wavelet shrinkage},
  author={Donoho, David L and Johnstone, Iain M},
  journal={{B}iometrika},
  volume={81},
  number={3},
  pages={425--455},
  year={1994},
  publisher={Oxford University Press}
}

@article{DJ98,
  title={Minimax estimation via wavelet shrinkage},
  author={Donoho, David L and Johnstone, Iain M},
  journal={{A}nn. {S}tatist.},
  volume={26},
  number={3},
  pages={879--921},
  year={1998},
  publisher={Institute of Mathematical Statistics}
}

@article{DRRS17,
	author = {Donnet, Sophie and Rivoirard, Vincent and Rousseau, Judith and Scricciolo, Catia},
	doi = {10.1214/15-BA986},
	fjournal = {{B}ayesian Analysis},
	issn = {1936-0975},
	journal = {{B}ayesian Anal.},
	mrclass = {62N02 (60G55 62F15 62G05)},
	mrnumber = {3597567},
	mrreviewer = {Alicja Jokiel-Rokita},
	number = {1},
	pages = {53--87},
	title = {Posterior concentration rates for counting processes with {A}alen multiplicative intensities},
	volume = {12},
	year = {2017}
}

@article{DG20,
  title={Multiscale functional inequalities in probability: Constructive approach},
  author={Duerinckx, Mitia and Gloria, Antoine},
  journal={Annales Henri Lebesgue},
  volume={3},
  pages={825--872},
  year={2020}
}

@book{E10,
  title={Partial Differential Equations},
  author={Evans, Lawrence C},
  volume={19},
  year={2010},
  publisher={American Mathematical Soc.}
}

@book{GvdV17,
  title={Fundamentals of nonparametric {B}ayesian inference},
  author={Ghosal, Subhashis and Van der Vaart, Aad W},
  year={2017},
  publisher={Cambridge University Press}
}

@article{GR25,
  title={Semiparametric {B}ernstein-von {M}ises theorems for reversible diffusions},
  author={Giordano, Matteo and Ray, Kolyan},
  journal={arXiv preprint arXiv:2505.16275},
  year={2025}
}

@book{GN16,
	author = {Gin\'e, Evarist and Nickl, Richard},
	isbn = {978-1-107-04316-9},
	mrclass = {62-02 (60B11 60F05 60F17 60G15 62E20 62Gxx)},
	mrnumber = {3588285},
	pages = {xiv+690},
	publisher = {Cambridge University Press, New York},
	title = {Mathematical foundations of infinite-dimensional statistical models},
	year = {2016}
}

@article{G23,
  title={{B}esov-{L}aplace priors in density estimation: optimal posterior contraction rates and adaptation},
  author={Giordano, Matteo},
  journal={Electron. J. Stat},
  volume={17},
  number={2},
  pages={2210--2249},
  year={2023},
  publisher={The Institute of Mathematical Statistics and the {B}ernoulli Society}
}

@article{GKR25,
  title={Nonparametric Bayesian intensity estimation for covariate-driven inhomogeneous point processes},
  author={Giordano, Matteo and Kirichenko, Alisa and Rousseau, Judith},
  journal={Bernoulli},
  volume={32},
  number={2},
  pages={1020--1044},
  year={2026},
  publisher={Bernoulli Society for Mathematical Statistics and Probability}
}

@article{GN20,
  title={Consistency of {B}ayesian inference with {G}aussian process priors in an elliptic inverse problem},
  author={Giordano, Matteo and Nickl, Richard},
  journal={Inverse Probl.},
  volume={36},
  number={8},
  pages={085001},
  year={2020},
  publisher={IOP Publishing}
}

@article{GR22,
  title={Nonparametric {B}ayesian inference for reversible multidimensional diffusions},
  author={Giordano, Matteo and Ray, Kolyan},
  journal={{A}nn. {S}tatist.},
  volume={50},
  number={5},
  pages={2872--2898},
  year={2022},
  publisher={Institute of Mathematical Statistics}
}

@article{G08,
	author = {Guan, Yongtao},
	fjournal = {Journal of the American Statistical Association},
	issn = {0162-1459},
	journal = {J. Amer. Statist. Assoc.},
	mrclass = {Expansion},
	mrnumber = {2528839},
	number = {483},
	pages = {1238--1247},
	title = {On consistent nonparametric intensity estimation for inhomogeneous spatial point processes},
	volume = {103},
	year = {2008}}

@article{HB15,
  title={Maximum a posteriori probability estimates in infinite-dimensional {B}ayesian Inverse Probl.},
  author={Helin, Tapio and Burger, Martin},
  journal={Inverse Probl.},
  volume={31},
  number={8},
  pages={085009},
  year={2015},
  publisher={IOP Publishing}
}

@book{HKPT12,
  title={Wavelets, approximation, and statistical applications},
  author={H{\"a}rdle, Wolfgang and Kerkyacharian, Gerard and Picard, Dominique and Tsybakov, Alexander},
  volume={129},
  year={2012},
  publisher={Springer Science \& Business Media}
}

@article{J93,
  title={Asymptotic normality of estimates in spatial point processes},
  author={Jensen, Jens Ledet},
  journal={Scand. J. Statist.},
  pages={97--109},
  year={1993},
  volume={20},
  number = {2},
  publisher={JSTOR}
}

@article{J81,
  title={Central limit theorem and convergence of empirical processes for stationary point processes},
  author={Jolivet, E},
  journal={Point processes and queuing problems (Colloqium, Keszthely, 1978)},
  volume={24},
  pages={117--161},
  year={1981},
  publisher={North-Holland Amsterdam-New York}
}

@article{JMS12,
  title={Pinpointing spatio-temporal interactions in wildfire patterns},
  author={Juan, Pablo and Mateu, Jorge and Saez, M},
  journal={Stochastic Environmental Research and Risk Assessment},
  volume={26},
  pages={1131--1150},
  year={2012}
}

@article{KvZ15,
	author = {Kirichenko, Alisa and van Zanten, Harry},
	fjournal = {Journal of Machine Learning Research (JMLR)},
	issn = {1532-4435},
	journal = {J. Mach. Learn. Res.},
	mrclass = {62F15 (60G15 60G55 68T05)},
	mrnumber = {3450529},
	pages = {2909--2919},
	title = {Optimality of {P}oisson processes intensity learning with {G}aussian processes},
	volume = {16},
	year = {2015}
}

@article{KLNS12,
  title={Sparsity-promoting {B}ayesian inversion},
  author={Kolehmainen, Ville and Lassas, Matti and Niinim{\"a}ki, Kati and Siltanen, Samuli},
  journal={Inverse Probl.},
  volume={28},
  number={2},
  pages={025005},
  year={2012},
  publisher={IOP Publishing}
}

@article{KPDO23,
    author = {Jonathan Koh and Fran{\c{c}}ois Pimont and Jean-Luc Dupuy and Thomas Opitz},
    title = {{Spatiotemporal wildfire modeling through point processes with moderate and extreme marks}},
    volume = {17},
    journal = {The Annals of Applied Statistics},
    number = {1},
    pages = {560 -- 582},
    year = {2023}}

@article {KS07,
    AUTHOR = {Kottas, Athanasios and Sans\'{o}, Bruno},
    TITLE = {{B}ayesian mixture modeling for spatial {P}oisson process intensities, with applications to extreme value analysis},
   	JOURNAL = {J. Statist. Plann. Inference},
  	FJOURNAL = {Journal of Statistical Planning and Inference},
    VOLUME = {137},
    YEAR = {2007},
    NUMBER = {10},
    PAGES = {3151--3163},
    ISSN = {0378-3758,1873-1171},
   	MRCLASS = {62M30 (62G07 62G32)},
  	MRNUMBER = {2365118},
    DOI = {10.1016/j.jspi.2006.05.022},
}

@book{K98,
	author = {Kutoyants, Yu. A.},
	isbn = {0-387-98562-X},
	mrclass = {62M09 (62G05)},
	mrnumber = {1644620},
	mrreviewer = {Paul I. Nelson},
	pages = {viii+276},
	publisher = {Springer-Verlag, New York},
	series = {Lecture Notes in Statistics},
	title = {Statistical inference for spatial {P}oisson processes},
	volume = {134},
	year = {1998}}

@article{LSS09,
  title={Discretization invariant {B}ayesian inversion and {B}esov space priors},
  author={Lassas, M and Saksman, E and Siltanen, S},
  journal={Inverse Probl. Imaging},
  volume={3},
  number={1},
  pages={87--122},
  year={2009},
  publisher={American Institute of Mathematical Sciences}
}

@article{LL99,
  title={Approximation, metric entropy and small ball estimates for {G}aussian measures},
  author={Li, Wenbo V and Linde, Werner},
  journal={The Annals of Probability},
  volume={27},
  number={3},
  pages={1556--1578},
  year={1999},
  publisher={Institute of Mathematical Statistics}
}

@article{L82,
	author = {Lo, Albert Y.},
	doi = {10.1007/BF00575525},
	fjournal = {Zeitschrift f\"{u}r Wahrscheinlichkeitstheorie und Verwandte Gebiete},
	issn = {0044-3719},
	journal = {Z. Wahrsch. Verw. Gebiete},
	mrclass = {62M07 (60G55)},
	mrnumber = {643788},
	mrreviewer = {Albert\ M.\ Liebetrau},
	number = {1},
	pages = {55--66},
	title = {{B}ayesian nonparametric statistical inference for {P}oisson point processes},
	volume = {59},
	year = {1982}}

@article{MSW98,
	author = {M{\o}ller, Jesper and Syversveen, Anne Randi and Waagepetersen, Rasmus Plenge},
	fjournal = {Scand. J. Stat. Theory and Applications},
	issn = {0303-6898,1467-9469},
	journal = {Scand. J. Statist.},
	mrclass = {62M09 (60G55 62C12 62F10 62M30)},
	mrnumber = {1650019},
	number = {3},
	pages = {451--482},
	title = {Log {G}aussian {C}ox processes},
	volume = {25},
	year = {1998}}

@book{N23,
  title={{B}ayesian Non-linear Statistical Inverse Probl.},
  author={Nickl, R.},
  isbn={9783985470532},
  series={Zurich Lectures in Advanced Mathematics},
  year={2023},
  publisher={EMS Press}
}

@book{RW06,
    author = {Rasmussen, Carl Edward and Williams, Christopher K. I.},
    title = {{G}aussian Processes for Machine Learning},
    publisher = {The MIT Press},
    year = {2005}
}

@book{R00,
  title={{G}aussian and non-{G}aussian linear time series and random fields},
  author={Rosenblatt, Murray},
  year={2000},
  publisher={Springer Science \& Business Media}
}

@article{RMC09,
	author = {Rue, H{\aa}vard and Martino, Sara and Chopin, Nicolas},
	doi = {10.1111/j.1467-9868.2008.00700.x},
	fjournal = {J. Roy. Statist. Soc. Ser. B},
	issn = {1369-7412,1467-9868},
	journal = {J. Roy. Statist. Soc. Ser. B},
	mrclass = {99-01},
	mrnumber = {2649602},
	number = {2},
	pages = {319--392},
	title = {Approximate {B}ayesian inference for latent {G}aussian models by using integrated nested {L}aplace approximations},
	volume = {71},
	year = {2009}}

@book {S86,
    AUTHOR = {Silverman, B. W.},
     TITLE = {Density estimation for statistics and data analysis},
    SERIES = {Monographs on Statistics and Applied Probability},
 PUBLISHER = {Chapman \& Hall, London},
      YEAR = {1986},
     PAGES = {x+175},
      ISBN = {0-412-24620-1},
   MRCLASS = {62G05},
  MRNUMBER = {848134},
MRREVIEWER = {David\ W.\ Scott},
       DOI = {10.1007/978-1-4899-3324-9},
}

@book{T02,
  title={Random heterogeneous materials: microstructure and macroscopic properties},
  author={Torquato, Salvatore},
  volume={16},
  year={2002},
  publisher={Springer}
}

@article{vdVvZ09,
author = {A. W. van der Vaart and J. H. van Zanten},
title = {{Adaptive {B}ayesian estimation using a {G}aussian random field with inverse Gamma bandwidth}},
volume = {37},
journal = {Ann. Statist.},
number = {5B},
publisher = {Institute of Mathematical Statistics},
pages = {2655--2675},
year = {2009},
}

@article{W07,
	author = {Waagepetersen, Rasmus Plenge},
	doi = {10.1111/j.1541-0420.2006.00667.x},
	fjournal = {Biometrics. Journal of the International Biometric Society},
	issn = {0006-341X,1541-0420},
	journal = {Biometrics},
	mrclass = {99-01},
	mrnumber = {2345595},
	number = {1},
	pages = {252--258},
	title = {An estimating function approach to inference for inhomogeneous {N}eyman-{S}cott processes},
	volume = {63},
	year = {2007},}

@article{YL11,
	author = {Yue, Yu Ryan and Loh, Ji Meng},
	doi = {10.1111/j.1541-0420.2010.01531.x},
	fjournal = {Biometrics. Journal of the International Biometric Society},
	issn = {0006-341X,1541-0420},
	journal = {Biometrics},
	mrclass = {99-01},
	mrnumber = {2829268},
	number = {3},
	pages = {937--946},
	title = {{B}ayesian semiparametric intensity estimation for inhomogeneous spatial point processes},
	volume = {67},
	year = {2011}}

@article{Z05,
  title={Towards reconciling two asymptotic frameworks in spatial statistics},
  author={Zhang, Hao and Zimmerman, Dale L},
  journal={{B}iometrika},
  volume={92},
  number={4},
  pages={921--936},
  year={2005},
  publisher={Oxford University Press}
}

@article{MOLL07,
author = {M{\o}ller, Jesper and Waagepetersen, Rasmus P.},
title = {Modern Statistics for Spatial Point Processes},
journal = {Scand. J. Statist.},
volume = {34},
number = {4},
pages = {643-684},
year = {2007}
}

@article{CDPS18,
author = {Chen, Victor and Dunlop, Matthew and Papaspiliopoulos, Omiros and Stuart, Andrew},
year = {2018},
month = {03},
pages = {},
title = {Robust MCMC Sampling with Non-{G}aussian and Hierarchical Priors in High Dimensions},
journal = {arXiv preprint arXiv:1803.03344}
}

\clearpage 
\begin{center}
\LARGE \textbf{Supplementary Material}
\end{center}
In this supplement, we present the proofs of all our results, additional simulations and another real data analysis.

\appendix

\counterwithin{equation}{section}
\counterwithin{figure}{section}
\counterwithin{table}{section}
%%%%%%%%%%%%%%%%%%%%%%%%%%%%%%%%%%%%%%%%%%%%%%%%%%%%%%%%%%%%%%%%%%%%%%%%%%%
\section{Further background material}
\label{App:AddMaterial}

In this appendix, we complement the theoretical results from the main paper presenting concrete examples of covariate random fields and Gaussian priors verifying the assumptions of the theorems.

%
%
%

%%%%%%%%%%%%%%%%%%%%%%%%%%%%%%%%%%%%%%%%%%%%%%%%%%%%%%%%%%%%%%%%%%%%%%%%%%
\subsection{Stationary and ergodic Gaussian covariates}\label{Subsec:GaussCov}

Gaussian processes are ubiquitous models for spatially correlated real-valued data in view of both their near-universal capability of describing physical phenomena and their methodological convenience \cite{RW06}. They provide the primary class of covariate random fields satisfying Condition \ref{Cond:ErgCov}, under the following mild assumptions on the associated covariance kernels.

\begin{example}\label{Ex:GaussCov}%%%%%%%%%%%%%%%%%%%%%%%%%%%%%%%%%%%%%%%
	For $h \in \{1,\dots,d\}$, let $\tilde Z^{(h)}:=(\tilde Z^{(h)}(x), \ x\in\R^D)$ be independent, almost surely locally bounded, centred and stationary Gaussian processes with integrable covariance functions, and assume that $\textnormal{Var}[Z^{(h)}(x)]=1$ for all $h \in \{1,\dots,d\}$ and $x\in\R^d$. Set
	\begin{equation}
		\label{Eq:PhiZ}
		Z(x) := [\phi(\tilde Z^{(1)}(x)), \dots, \phi(\tilde Z^{(d)}(x))], \qquad x\in\R^D,
	\end{equation}
	where $\phi$ is the standard normal cumulative distribution function (c.d.f.). It follows that $Z$ is stationary, with uniform stationary distribution $\nu_Z = \textnormal{Un}([0,$ $1]^d)$. Further, by Proposition D.1 in \cite{GKR25}, provided that $\Wcal_n$ is as in \eqref{Eq:SpatialWn}, $Z$ enjoys the concentration inequality \eqref{Eq:ConcIneq} with $K_2 = 1/(c K_1^2)$ for some numerical constant $c>0$ that only depends on $d$ and the covariance functions.
\end{example}%%%%%%%%%%%%%%%%%%%%%%%%%%%%%%%%%%%%%%%%%%%%%%%%%%%%%%%%%%%%

For $Z$ as in Example \ref{Ex:GaussCov}, the point pattern $N^{(n)}$ from \eqref{Eq:PointProc} is a `nonparametric version' of the popular log-Gaussian Cox process \cite{MSW98}, which prescribes a parametric covariate-based intensity function $\rho(z) = \exp(\beta^\top z)$, $z\in\R^d$, for some $\beta\in\R^d$. The assumption that the random vectors $(\tilde Z^{(1)}(x),\dots,\tilde Z^{(d)}(x))$ be standard $d$-variate Gaussian for all $x\in\R^d$ is without loss of generality (in that other covariance structures could be accommodated as well) and reflects the common operation of standardising the covariates before the analysis, e.g., ~\cite[Section 3.2]{G08}. The integrability of the covariance functions is a standard condition that ensures that $Z$ is ergodic; it is satisfied as long as the correlation between covariates at distant locations decays sufficiently fast, which is a common scenario in practice, e.g.,~\cite[p.~58]{C15}. Lastly, the transformation \eqref{Eq:PhiZ} represents a concrete instance of the pre-processing map $\Phi$ mentioned after Condition \ref{Cond:ErgCov}. Similar compactification steps are widespread in spatial statistics; for example, the popular $\texttt{R}$ package \texttt{spatstat} \cite{BRT16} for kernel-based intensity estimation also automatically performs a rescaling to $[0,1]^d$. 

%
%
%

%%%%%%%%%%%%%%%%%%%%%%%%%%%%%%%%%%%%%%%%%%%%%%%%%%%%%%%%%%%%%%%%%%%%%%%%%%%
\subsection{Poisson random tessellations}
\label{App:OtherCov}

The second major class of stationary and ergodic covariate processes satisfying Condition \ref{Cond:ErgCov} are piecewise constant constructions based on random tessellations. These are suited to model discrete heterogeneous objects that appear in a variety of applications, such as the material sciences \cite{T02}.

\begin{example}\label{Ex:PoissTess}%%%%%%%%%%%%%%%%%%%%%%%%%%%%%%%%%%%%%%%%
	Let $\Xi := (\xi_i, \ i\in\N)$ be a standard Poisson point process on $\R^D$, and let $(\Ccal_i, \ i\in\N)$ be the random partition of $\R^D$ given by the associated Voronoi tessellation, that is, 
	$$
	\Ccal_i := \left\{ x\in \R^D : |x - \xi_i| = \inf_{j\in\N }|x - \xi_j|\right\}.
	$$
	For $\nu_Z$ an absolutely continuous probability measure supported on $[0,1]^d$ with bounded p.d.f., set
	$$
	Z(x) = \sum_{i=1}^\infty\zeta_i1_{\Ccal_i}(x), 
	\qquad \zeta_i\iid \nu_Z,
	\qquad x\in\R^D.
	$$
	It follows that $Z$ is stationary, with stationary distribution equal to $\nu_Z$. Further, by Lemma D.6 in \cite{GKR25}, provided that $\Wcal_n$ is as in \eqref{Eq:SpatialWn}, $Z$ enjoys the concentration inequality \eqref{Eq:ConcIneq} with $K_2 = 1/(c K_1)$ for some numerical constant $c>0$.
\end{example}%%%%%%%%%%%%%%%%%%%%%%%%%%%%%%%%%%%%%%%%%%%%%%%%%%%%%%%%%%%%%%

The covariate random field $Z$ from Example \ref{Ex:PoissTess} is piecewise constant over the cells of the random Voronoi tessellations $(\Ccal_i, \ i\in\N)$, with cell-wise values $(\zeta_i, \ i\in\N)$ drawn at random from the stationary distribution $\nu_Z$. The assumption that $\nu_Z$ be supported on $[0,1]^d$ is, again, for concreteness; any other compact subset of $\R^d$ could be taken as well. Note that, in view of the standard Poisson assumption on $\Xi$, large subsets in the random tessellation occur with small probability. This is the key feature determining the ergodicity of $Z$; see \cite[Section 3]{DG20} for further details.

%
%
%

%%%%%%%%%%%%%%%%%%%%%%%%%%%%%%%%%%%%%%%%%%%%%%%%%%%%%%%%%%%%%%%%%%%%%%%%%%%
\subsection{Example of Gaussian priors}
\label{App:GPs}

We provide details on two classes of Gaussian priors satisfying the requirements of Condition \ref{Cond:GPCondition}. We start with the Gaussian wavelet series priors employed in the simulation studies of Section \ref{Sec:Simulations} and Appendix \ref{Suppl:AddSimul}.

\begin{example}[Gaussian wavelet series priors]\label{Ex:GPWav}%%%%%%%%%%%%%
	For $(\psi_\ell, \ \ell\in\N)$ a (single-index re-ordering of an) orthonormal wavelet basis of $L^2([0,1]^d)$, see ~Section \ref{Subsec:Notation} for details, and for any $\alpha>1+d$, let $\Pi$ be the law of the random function
	\begin{equation}
		\label{Eq:GPWav}
		W(z) = \sum_{\ell=1}^\infty \ell^{-\alpha/d} w_\ell \psi_\ell(z), 
		\quad z\in[0,1]^d,
		\quad w_\ell\iid N(0,1).
	\end{equation}
	The RKHS of $\Pi$ is equal to
	$$
	\Hcal
	= \left\{f=\sum_{\ell=1}^\infty f_\ell \psi_\ell : \| f \|^2_{\Hcal}
	:=\sum_{\ell=1}^\infty \ell^{ 2\alpha/d}|f_\ell|^2<\infty \right\},
	$$
	see ~\cite[Section 11.4.5]{GvdV17}, so that it holds, with norm equivalence, that $\Hcal = B^\alpha_2([0,1]^d) = H^\alpha([0,1]^d)$; see Section \ref{Subsec:Notation}. Further, a simple computation yields that $\Enorm^\Pi[\|W\|_{H^{\tilde\alpha}}^2]<\infty$ for all $\tilde\alpha < \alpha - d/2$. Taking any $1+d/2<\tilde \alpha <\alpha - d/2$, this implies that $\Pi(C^1([0,1]^d)) = 1$ by the Sobolev embedding (e.g.,~\cite[p.~370]{GN16}), which shows that $\Pi$ satisfies Condition \ref{Cond:GPCondition}. 
\end{example}%%%%%%%%%%%%%%%%%%%%%%%%%%%%%%%%%%%%%%%%%%%%%%%%%%%%%%%%%%%%%%

We next consider the widely used class of Matérn processes.

\begin{example}[Matérn process priors]\label{Ex:Matern}%%%%%%%%%%%%%%%%%%%%
	For fixed $\alpha>1+d/2$ and $\ell>0$, let $W=(W(z),\ z\in [0,1]^d)$ be the centred and stationary Gaussian process with Matérn covariance kernel
	\begin{equation*}
		C_{\textnormal{Mat}}(z,y) := \frac{2^{1-\alpha}}{\Gamma(\alpha)}\left(\frac{|z-y|\sqrt {2 \alpha}}{\ell}\right)^\alpha
		B_\alpha \left(\frac{|z-y|\sqrt{2\alpha}}{\ell}\right),
		\qquad z,y\in[0,1]^d,
	\end{equation*}
	where $\Gamma$ denotes the Gamma function and $B_\alpha$ is the modified Bessel function of the second kind. By Lemma I.4 in \cite{GvdV17}, $W$ admits a version whose sample paths are almost surely in $C^{\tilde\alpha}([0,1]^d)$ for any $\tilde\alpha<\alpha-d/2$. Taking any $1<\tilde\alpha<\alpha - d/2$, this implies that the law $\Pi$ of such version defines a (centred) Gaussian Borel probability measure supported on $C^1([0,1]^d)$ in view of Lemma I.7 in \cite{GvdV17}.

	Furthermore, by the results in Section 11.4.4 of \cite{GvdV17}, the RKHS $\Hcal$ of $W$ is equal to the set of restrictions to $[0,1]^d$ of functions in the Sobolev space $H^\alpha(\R^d)$, defined via the Fourier transform. This coincides, with norm equivalence, to $H^\alpha([0,1]^d)$, cf.~p.~351 and 371 in \cite{GN16}. This shows that $\Pi$ satisfies Condition \ref{Cond:GPCondition}. 
\end{example}%%%%%%%%%%%%%%%%%%%%%%%%%%%%%%%%%%%%%%%%%%%%%%%%%%%%%%%%%%%%%%%

%%%%%%%%%%%%%%%%%%%%%%%%%%%%%%%%%%%%%%%%%%%%%%%%%%%%%%%%%%%%%%%%%%%%%%%%%
\section{A general posterior contraction result and proofs of Theorems \ref{Theo:GPRates} and \ref{Theo:LaplRates}}
\label{Sec:Proofs}

The proofs of Theorems \ref{Theo:GPRates} and \ref{Theo:LaplRates} are based on the following general posterior contraction result, holding under abstract prior conditions similar to the standard ones from Bayesian nonparametrics, cf.~\cite[Chapter 8]{GvdV17}. It refines ideas from \cite{GKR25}, involving a preliminary step where asymptotic concentration is obtained in the empirical `testing' distance \eqref{Eq:EmpDist}, to be subsequently related to the $L^1([0,1]^d,\nu_Z)$-metric via the exponential concentration inequality \eqref{Eq:ConcIneq}.

\begin{theorem}\label{Theo:GenTheo}%%%%%%%%%%%%%%%%%%%%%%%%%%%%%%%%%%%%%%
	
	For $\rho_0\in\Rcal\cap C^1([0,1]^d)$, consider data $(N^{(n)},Z^{(n)})\sim P^{(n)}_{\rho_0}$ from \eqref{Eq:PointProc} with $\rho=\rho_0$ and with $Z$ satisfying Condition \ref{Cond:ErgCov}. Let $\Pi_n$ be a (possibly $n$-dependent) probability measure supported on $\Rcal\cap C^1([0,1]^d)$ and assume that for some positive sequence $\varepsilon_n\to0$ such that $n\varepsilon_n^2\to\infty$ as $n\to\infty$, 
	\begin{equation}
		\label{Eq:SmallBall}
		\Pi(\rho : \|\rho - \rho_0\|_{L^\infty}\ge \varepsilon_n)
		\ge e^{-L_1n\varepsilon_n^2},
	\end{equation}
	for some $L_1>0$. Further assume that for all $L_2>0$ there exist measurable sets $\Rcal_n\subseteq \Rcal\cap C^1([0,1]^d)$ such that
	\begin{equation}
		\label{Eq:Sieves}
		\Pi_n(\Rcal_n^c)
		\le e^{-L_2n\varepsilon_n^2},
		\qquad \ln \Ncal(\varepsilon_n;\Rcal_n,\|\cdot\|_{L^1})\le L_3n\varepsilon_n^2,
	\end{equation}
	and $\|\rho\|_{C^1}\le L_4$ for all $\rho\in\Rcal_n$ and some $L_3,L_4>0$. Then, for $M>0$ large enough, as $n\to\infty$,
	$$
	\Enorm_{\rho_0}^{(n)}
	\Bigg[\Pi_n\Big(\rho : \|\rho - \rho_0\|_{L^1([0,1]^d,\nu_Z)} 
	> M\varepsilon_n
	\Big| N^{(n)},Z^{(n)}\Big)\Bigg]
	\to 0.
	$$
\end{theorem}%%%%%%%%%%%%%%%%%%%%%%%%%%%%%%%%%%%%%%%%%%%%%%%%%%%%%%%%%%%%

Condition \eqref{Eq:Sieves} involves the covering number $\Ncal(\varepsilon_n;\Rcal_n,\|\cdot\|_{L^1})$, defined as the smallest number of $L^1$-balls of radius $\varepsilon_n$ required to contain $\Rcal_n$ in their union. The main improvement of Theorem \ref{Theo:GenTheo} relative to the general result from Section 3.1 in \cite{GKR25} consists in expressing the `metric entropy' inequality in \eqref{Eq:Sieves} with respect to the $L^1$-norm instead of the stronger $L^\infty$-norm from eq.~(3.3) in \cite{GKR25}. This is crucial for deriving optimal posterior contraction rates for Besov-Laplace priors, for which $L^\infty$-complexity bounds are known to be too restrictive, cf.~the discussion after Theorem 1 in \cite{G23}.

Given the abstract result, the proofs of Theorems \ref{Theo:GPRates} and \ref{Theo:LaplRates} boil down to verifying the two prior conditions \eqref{Eq:SmallBall} and \eqref{Eq:Sieves} above, for which we use well-established techniques for rescaled Gaussian and Besov-Laplace priors. We include this verification in \ref{Subsec:ProofGPRates} and \ref{Subsec:ProofLaplRates} for completeness and the convenience of the reader.

\begin{proof}%%%%%%%%%%%%%%%%%%%%%%%%%%%%%%%%%%%%%%%%%%%%%%%%%%%%%%%%%%%%
	We build on the arguments to prove Theorems 3.1 and 3.12 in \cite{GKR25}, first showing that, for $\tilde\Rcal_n := \{\rho \in\Rcal_n: \|\lambda^{(n)}_\rho - \lambda^{(n)}_{\rho_0}\|_{L^1(\Wcal_n)} \le \tilde Mn\varepsilon_n\}$,
	\begin{equation}
		\label{Eq:EmpricalRates}
		\Enorm_{\rho_0}^{(n)}
		\left[\Pi(\tilde\Rcal_n^c | N^{(n)},Z^{(n)})\right]\to 0
	\end{equation}
	as $n\to\infty$ for sufficiently large $\tilde M>0$ to be chosen below. We have,
	\begin{equation}
		\label{Eq:PostRn}
		\Pi(\tilde\Rcal_n^c | N^{(n)},Z^{(n)}) 
		=
		\frac{\int_{\tilde\Rcal_n^c} L^{(n)}(\rho)/L^{(n)}(\rho_0)d\Pi(\rho)}
		{B^{(n)}},
	\end{equation}
	by Bayes' formula \eqref{Eq:Post}, with $B^{(n)}$ the normalisation constant, which, in view of the small ball lower bound \eqref{Eq:SmallBall}, using Lemma 8.21 of \cite{GvdV17} jointly with Lemma A.1 (and Remark 3.2) in \cite{GKR25}, satisfies
	\begin{equation}
		\label{Eq:Denom}
		P_{\rho_0}^{(n)}
		\left( B^{(n)}\le e^{- (c_1 + L_1)n\varepsilon_n^2} \right) 
		= O((n \varepsilon_n^2)^{-1} ), 
	\end{equation}
	as $n\to\infty$ for some constant $c_1>0$ only depending on $\|\rho_0\|_{L^\infty}$. Hence,
	\begin{equation*}
			\Enorm_{\rho_0}^{(n)}[\Pi(\tilde\Rcal_n^c | N^{(n)},Z^{(n)}) ] \le 	\Enorm_{\rho_0}^{(n)}\left[\Pi(\tilde\Rcal_n^c | N^{(n)},Z^{(n)}) 
			1_{\{ B^{(n)} > e^{-(c_1+L_1)n\varepsilon_n^2}\}}\right]
			+O((n\varepsilon_n^2)^{-1}),
	\end{equation*}
	as $n\to\infty$. Next, we upper bound the expectation in the last line by
	\begin{equation*}
		\begin{split}
			\Enorm_{\rho_0}^{(n)}&
			\left[\Pi(\rho\in\Rcal_n :
			\|\lambda^{(n)}_\rho - \lambda^{(n)}_{\rho_0}\|_{L^1(\Wcal_n)} > \tilde Mn\varepsilon_n  | N^{(n)},Z^{(n)}) 
			1_{\{ B^{(n)} > e^{-(c_1 + L_1 )n\varepsilon_n^2} \}}\right]\\
			&+\Enorm_{\rho_0}^{(n)}\left[\Pi(\Rcal_n^c | N^{(n)},Z^{(n)}) 
			1_{\{ B^{(n)} > e^{-(c_1 + L_1 )n\varepsilon_n^2} \}}\right]
		\end{split}
	\end{equation*}
	with the last term being smaller than
	\begin{equation*}
		\begin{split}
			e^{(c_1 + L_1 )n\varepsilon_n^2}
			\int_{\Rcal_n^c}
			\Enorm_{\rho_0}^{(n)}\left[ L^{(n)}(\rho)/L^{(n)}(\rho_0)\right] d\Pi(\rho)
			&= e^{(c_1 + L_1 )n\varepsilon_n^2}
			\Pi(\Rcal_n^c) = o(1),
		\end{split}
	\end{equation*}
	as $n\to\infty$, having used Fubini's theorem, the identity  $\Enorm_{\rho_0}^{(n)} [ L^{(n)}(\rho)/L^{(n)}(\rho_0)]$ $= \Enorm^{(n)}_\rho[1]=1$ and assumption \eqref{Eq:Sieves}. Next, by Lemma \ref{Lem:GlobalTests}, for all sufficiently large $\tilde M>0$ and $n\in\N$, there exists a test $\phi_n$ such that
	\begin{equation}
		\label{Eq:TypeI}
		\max\left\{\Enorm_{\rho_0}^{(n)}[\phi_n],\sup_{\rho\in\Rcal_n : \|\lambda^{(n)}_\rho - \lambda^{(n)}_{\rho_0}\|_{L^1(\Wcal_n)}\ge \tilde M n\varepsilon_n}
		\Enorm^{(n)}_\rho[1-\phi_n ]\right\}
		\le 3e^{-c_2 \tilde M^2 n\varepsilon_n^2},
	\end{equation}
	for some fixed constant $c_2>0$ that only depends on $\|\rho_0\|_{C^1}$. The proof of \eqref{Eq:EmpricalRates} is then concluded by noting that, provided that $\tilde M$ is large enough,
	\begin{equation*}
		\begin{split}
			&\Enorm_{\rho_0}^{(n)}
			\left[\Pi(\rho\in\Rcal_n :
			\|\lambda^{(n)}_\rho - \lambda^{(n)}_{\rho_0}\|_{L^1(\Wcal_n)} > \tilde Mn\varepsilon_n  | N^{(n)},Z^{(n)}) 
			1_{\{ B^{(n)} > e^{-(c_1 + L_1 )n\varepsilon_n^2} \}}\right]
			\\
			& \ \le
			\Enorm_{\rho_0}^{(n)}[\phi_n] +e^{(c_1 + L_1 )n\varepsilon_n^2}\Enorm_{\rho_0}^{(n)}
			\left[\int_{\{\rho\in\Rcal_n :
				\|\lambda^{(n)}_\rho - \lambda^{(n)}_{\rho_0}\|_{L^1(\Wcal_n)} > \tilde Mn\varepsilon_n\}}
			\frac{L^{(n)}(\rho)}{L^{(n)}(\rho_0)}(1-\phi_n)d\Pi(\rho)
			\right]\\
			&\le 3e^{- c_2 \tilde M^2 n\varepsilon_n^2}+ 3e^{-( c_2\tilde M^2 - c_1 - L_1 )n\varepsilon_n^2}
			=o(1),
		\end{split}
	\end{equation*}
	having used \eqref{Eq:TypeI} and Fubini's theorem to deduce that
	\begin{align*}
		\Enorm_{\rho_0}^{(n)}&
		\left[\int_{\{\rho\in\Rcal_n :
			\|\lambda^{(n)}_\rho - \lambda^{(n)}_{\rho_0}\|_{L^1(\Wcal_n)} > \tilde Mn\varepsilon_n\}}
		\frac{L^{(n)}(\rho)}{L^{(n)}(\rho_0)}(1-\phi_n)d\Pi(\rho)
		\right] \\
        & =\int_{\{\rho\in\Rcal_n :
			\|\lambda^{(n)}_\rho - \lambda^{(n)}_{\rho_0}\|_{L^1(\Wcal_n)} > \tilde Mn\varepsilon_n\}}
		\Enorm_{\rho}^{(n)}[1 - \phi_n]d\Pi(\rho)
		\le 3e^{-c_2\tilde M^2 n\varepsilon_n^2}.
	\end{align*}

	For the second part of the proof, we show that, for $\Tcal_n :=\{\rho\in \Rcal_n:\|\rho - \rho_0\|_{L^1([0,1]^d,\nu_Z)} \le M\varepsilon_n\}$, with sufficiently large $M>0$ to be chosen below, $\Enorm_{\rho_0}^{(n)} [\Pi(\Tcal_n^c | N^{(n)}, Z^{(n)}) ] \to0$ as $n\to\infty$. To do so, note that by \eqref{Eq:EmpricalRates},
	\begin{align*}
		\Pi(\Tcal_n^c | N^{(n)}, Z^{(n)}) 
		&= \frac{ \int_{\Tcal^c_n \cap \tilde \Rcal_n} L^{(n)}(\rho)/L^{(n)}(\rho_0)d\Pi(\rho) }
		{ B^{(n)}}  
		+ o_{P^{(n)}_{\rho_0}}(1),
	\end{align*}
	with $B^{(n)}$ the same normalisation constant from \eqref{Eq:PostRn}. Arguing similarly to the previous steps, using \eqref{Eq:Denom} and Fubini's theorem, we then obtain that
	\begin{equation*}
		\Enorm_{\rho_0}^{(n)} \left[\Pi(\Tcal_n^c | N^{(n)}, Z^{(n)})\right] \leq e^{  (c_1 + L_1) n \varepsilon_n^2 }  
		\int_{\Tcal_n^c \cap  \Rcal_n} 
		P_{Z^{(n)}} (\|\lambda^{(n)}_\rho - \lambda^{(n)}_{\rho_0}
		\|_{L^1(\Wcal_n)} \le \tilde Mn\varepsilon_n) d\Pi(\rho)+ o(1).
	\end{equation*}
	Fix any $\rho\in \Tcal_n^c \cap \Rcal_n$. Then, if $\|\lambda^{(n)}_\rho - \lambda^{(n)}_{\rho_0}\|_{L^1(\Wcal_n)} \le \tilde Mn\varepsilon_n $, necessarily 
	$$
	\Delta^{(n)}(\rho) := \| \rho - \rho_0\|_{L^1([0,1]^d,\nu_Z)} 
	- \frac{1}{n}\|\lambda^{(n)}_\rho - \lambda^{(n)}_{\rho_0}
	\|_{L^1(\Wcal_n)} > (M-\tilde M)\varepsilon_n
	\geq \frac{\tilde M}{2}\varepsilon_n  
	$$ 
	upon taking $M > 2\tilde M$. Thus, the expectation of interest is smaller than
	\begin{align*}
		e^{  (c_1+L_1) n \varepsilon_n^2 }  \int_{\Tcal_n^c \cap R_n}
		P_{Z^{(n)}} (\Delta^{(n)}(\rho) > \tilde M\varepsilon_n/2)d\Pi(\rho)+ o(1)  .  
	\end{align*}
	The concentration inequality \eqref{Eq:ConcIneq} applied with $f := |\rho - \rho_0|- n^{-1}\|\lambda^{(n)}_\rho - \lambda^{(n)}_{\rho_0}\|_{L^1(\Wcal_n)}$, for $\rho\in \Rcal_n$, whose (weak) gradient satisfies $ \| \nabla f  \|_{L^\infty([0,1]^d;\R^d)}\le \| \rho  \|_{C^1}+\|  \rho_0  \|_{C^1}\lesssim 1$ now gives that
	$$
	\sup_{\rho\in \Tcal_n^c \cap\Rcal_n}
	P_{Z^{(n)}} \left(\Delta^{(n)}(\rho) > \tilde M\varepsilon_n/2\right)
	\leq e^{- c_3 \tilde M^2 n\varepsilon_n^2 }
	$$
	for some $c_3>0$. The proof is then concluded taking $\tilde M>0$ large enough and combining the last two displays.
\end{proof}%%%%%%%%%%%%%%%%%%%%%%%%%%%%%%%%%%%%%%%%%%%%%%%%%%%%%%%%%%%%%%

The following key lemma shows the existence of consistent tests with uniformly controlled type-II error probabilities. It refines, under Condition \ref{Cond:ErgCov}, Lemma A.3 in \cite{GKR25} by replacing the $L^\infty$-metric entropy inequality assumed therein with the weaker $L^1$-complexity bound from \eqref{Eq:Sieves}.

\begin{lemma}\label{Lem:GlobalTests}%%%%%%%%%%%%%%%%%%%%%%%%%%%%%%%%%%%%%
	Let $\rho_0\in\Rcal\cap C^1([0,1]^d)$, let $Z$ satisfy Condition \ref{Cond:ErgCov}, and let the sets $\Rcal_n$ be as in the statement of Theorem \ref{Theo:GenTheo} for some $L_2,L_3,L_4>0$ and some positive sequence $\varepsilon_n\to0$ such that $n\varepsilon_n^2\to\infty$ as $n\to\infty$. Then, for all sufficiently large $M>0$ and $n\in\N$, there exists a test $\phi_n$ such that 
	$$
	\max\left\{\Enorm_{\rho_0}^{(n)}[\phi_n]
	,\sup_{\rho\in\Rcal_n : \|\lambda_\rho^{(n)} - \lambda_{\rho_0}^{(n)}\|_{L^1(\Wcal_n)}\ge M n\varepsilon_n}
	\Enorm^{(n)}_\rho[1-\phi_n ]\right\}
	\le 3e^{-CM^2 n\varepsilon_n^2}
	$$
	for some constants $C>0$ that only depend on $\|\rho\|_{C^1}$.
\end{lemma}%%%%%%%%%%%%%%%%%%%%%%%%%%%%%%%%%%%%%%%%%%%%%%%%%%%%%%%%%%%%%%

\begin{proof}%%%%%%%%%%%%%%%%%%%%%%%%%%%%%%%%%%%%%%%%%%%%%%%%%%%%%%%%%%%%
	For some sufficiently large $M>0$ to be chosen below, cover the set $\{\rho\in\Rcal_n: \|\lambda^{(n)}_\rho - \lambda^{(n)}_{\rho_0}\|_{L^1(\Wcal_n)}\ge M n\varepsilon_n\}$ via $L^1([0,1],\nu_Z)$-balls of radius $M\varepsilon_n/4$ and centres $(\rho_j, \ j\in\{1,\dots,$ $\Ncal^{(n)}\}$, where, in view of \eqref{Eq:Sieves} and the boundedness assumption on the p.d.f.~of $\nu_Z$, 
	\begin{equation}
		\label{Eq:NBound}
		\Ncal^{(n)} \le \Ncal(\varepsilon_n;\Rcal_n,\|\cdot\|_{L^1})
		\le e^{L_3n\varepsilon_n^2},
	\end{equation}
	provided that $M$ is large enough. For $\ell \in \{1,\dots,\Ncal^{(n)}\}$, by Lemma A.2 in \cite{GKR25}, there exists a test $\phi_{\rho_j}$ such that for all sufficiently large $n$
	\begin{align*}
		\Enorm_{\rho_0}^{(n)}[\phi_{\rho_j}|Z^{(n)}]
		\le
		2e^{-c_1 M^2n\varepsilon_n^2},
	\end{align*}
	for some constant $c_1>0$ that only depends on $\|\rho_0\|_{L^\infty}$, as well as
	\begin{align*}
		\sup_{\rho:\|\lambda^{(n)}_\rho-\lambda^{(n)}_{\rho_j}\|_{L^1(\Wcal_n)}
			\le\|\lambda^{(n)}_{\rho_j} - \lambda^{(n)}_{\rho_0}\|_{L^1(\Wcal_n)}/2}
		\Enorm^{(n)}_\rho[1-\phi_{\rho_j}|Z^{(n)}]
		&\le 2e^{-c_1 M^2n\varepsilon_n^2}.
	\end{align*}
	Set $\phi_n:=\max_{j=1,\dots,\Ncal^{(n)}}\phi_{\rho_j}$. Then, for all sufficiently large $n$, in view of \eqref{Eq:NBound},
	\begin{align*}
		\Enorm_{\rho_0}^{(n)}[\phi_n|Z^{(n)}]
		\le \sum_{j=1}^{\Ncal^{(n)}}\Enorm_{\rho_0}^{(n)}[\phi_{\rho_j}|Z^{(n)}]
		\le 2\Ncal^{(n)}e^{-c_1 M^2n\varepsilon_n^2}
		\le 2e^{-(c_1 M^2 - L_3)n\varepsilon_n^2},
	\end{align*}
	which, provided that $M^2>\max(L_3/c_1,1)$, concludes the derivation of the first claim. Further, since for for each $\rho\in\Rcal_n$ with $\|\lambda^{(n)}_\rho - \lambda^{(n)}_{\rho_0}\|_{L^1(\Wcal_n)}>M n \varepsilon_n$ there exists, by construction, a centre $\rho_j$ with $\|\rho - \rho_j\|_{L^1([0,1]^d,\nu_Z)}\le M\varepsilon_n/4$, on the event
	$$
	\Omega_\rho :=\left\{\left| \frac{1}{n}\|\lambda^{(n)}_{\rho} - \lambda^{(n)}_{\rho_j}\|_{L^1(\Wcal_n)} - \|\rho - \rho_j\|_{L^1([0,1]^d,\nu_Z)} \right|
	<\frac{M\varepsilon_n}{4}\right\},
	$$
	we have
	\begin{equation*}
		\|\lambda^{(n)}_\rho - \lambda^{(n)}_{\rho_j}\|_{L^1(\Wcal_n)} \le \frac{Mn\varepsilon_n}{4} 
		+ n\|\rho - \rho_j\|_{L^1([0,1]^d,\nu_Z)} \le \frac{1}{2} Mn\varepsilon_n
		\le \frac{1}{2} \|\lambda^{(n)}_{\rho_j} -
		\lambda^{(n)}_{\rho_0}\|_{L^1(\Wcal_n)},
	\end{equation*}
	and therefore
	$$
	1_{\Omega_\rho} \Enorm_\rho^{(n)}[
	1 - \phi_{\rho_j}|Z^{(n)}]
	\le 2e^{-c_1 M^2n\varepsilon_n^2}.
	$$
	Conclude by the tower property of conditional expectation that
	\begin{align*}
		\Enorm_\rho^{(n)}[1 - \phi_n]
		%    &\le \Enorm_\rho^{(n)}[(1 - \phi_n)1_{\Omega_\rho}]
		%    + P_{Z^{(n)}}(\Omega_\rho^c) \\
		&\le \Enorm_\rho^{(n)}[(1 - \phi_{\rho_j})1_{\Omega_\rho}]
		+ P_{Z^{(n)}}(\Omega_\rho^c)=\Enorm_\rho^{(n)}[1_{\Omega_\rho} \Enorm_\rho^{(n)}[
		1 - \phi_{\rho_j}|Z^{(n)}]]
		+ P_{Z^{(n)}}(\Omega_\rho^c)\\
		&\le 2e^{-c_1 M^2n\varepsilon_n^2}
		+ P_{Z^{(n)}}(\Omega_\rho^c).
	\end{align*}
	By the exponential concentration inequality \eqref{Eq:ConcIneq}, the latter probability is upper bounded by $e^{-c_2 M^2n\varepsilon_n^2}$ for some constant $c_2>0$ that only depends on $\|\rho_0\|_{C^1}$. Combined with the previous display and the fact that $\|\rho_0\|_{C^1}\le L_4$ for all $\rho\in\Rcal_n$ by assumption, this proves the second claim.
\end{proof}%%%%%%%%%%%%%%%%%%%%%%%%%%%%%%%%%%%%%%%%%%%%%%%%%%%%%%%%%%%%%%

%We conclude this Section by showing that conditions \eqref{Eq:SmallBall} and \eqref{Eq:Sieves} in Theorem \ref{Theo:GenTheo} are satisfied by both prior choices in Sections \ref{Subsec:GPRates} and \ref{Subsec:LaplRates}. Accordingly, we prove Theorems \ref{Theo:GPRates} and \ref{Theo:LaplRates}, assuring optimal posterior contraction rates for both prior choices.

%%%%%%%%%%%%%%%%%%%%%%%%%%%%%%%%%%%%%%%%%%%%%%%%%%%%%%%%%%%%%%%%%%%%%%%%%
\subsection{Proof of Theorem \ref{Theo:GPRates}}
\label{Subsec:ProofGPRates}

Set $\varepsilon_n:= n^{-\alpha/(2\alpha+d)}$. For $\Pi$ a base Gaussian probability measure satisfying Condition \ref{Cond:GPCondition}, let $\tilde \Pi_n$ be the law of the random function $\tilde W_n := (\sqrt n\varepsilon_n)^{-1} W$, $W\sim \Pi$. Lemma 5.2 in \cite{GR22} implies that for all $c_1>0$ there exists sufficiently large $c_2>0$ such that the sets
$$
\Ucal_n:=\{ u = u_1 + u_2 : \|u_1\|_{L^2}\le \varepsilon_n,
\|u_2\|_{H^\alpha}\le c_2, \|u\|_{C^1}\le c_2\}
$$
satisfy $\tilde\Pi_n(\Ucal_n^c)\le e^{-c_1n\varepsilon_n^2}$. Let $\Rcal_n:=\left\{\eta\circ u,\ u\in \Ucal_n\right\}$; then $\Pi_n (\Rcal_n^c)\le \tilde\Pi_n(\Ucal_n^c) \le e^{-c_1n\varepsilon_n^2}$. Further, since $\eta$ is smooth by assumption, it is locally Lipschitz, and $\| \eta\circ u_1 - \eta\circ u_2\|_{L^2}\lesssim \|u_1 -  u_2\|_{L^2}$ for all $u\in\Ucal_n$ and some multiplicative constant that only depends on $c_2$. Thus,
$$
\ln \Ncal(\varepsilon_n;\Rcal_n,\|\cdot\|_{L^1})
\le \ln \Ncal(\varepsilon_n;\Rcal_n,\|\cdot\|_{L^2})
\lesssim \ln \Ncal(\varepsilon_n;\Ucal_n,\|\cdot\|_{L^2}).
$$
The latter is upper bounded by a multiple of $\varepsilon_n^{-d/\alpha}=n\varepsilon_n^2$ in view of the metric entropy estimate in Theorem 4.3.36 of \cite{GN16}, and since $\Ucal_n$ is contained in a $L^2$-enlargement of radius $\varepsilon_n$ of the ball $\{u\in H^\alpha([0,1]^d):\|u\|_{H^\alpha}\le c_2\}$. This verifies the `sieve condition' \eqref{Eq:Sieves}. Also, $\|\eta\circ u\|_{C^1}\lesssim \|u\|_{C^1}\lesssim 1$ for all $u\in\Ucal_n$.

Lastly, recalling that $w_0\in H^\beta([0,1]^d)\subseteq H^\alpha([0,1]^d)=\Hcal$ since $\alpha\le\beta$, by Corollary 2.6.18 of \cite{GN16}, for all sufficiently large $n$,
\begin{align*}
	& \Pi_n(\rho:\|\rho - \rho_0\|_{L^\infty} \le \varepsilon_n)\\ 
     &\ge 
     \tilde \Pi_n\left(w:\|w - w_0\|_{L^\infty} 
	\le \frac{1}{2}\varepsilon_n\right) \\
    & \ge 
     e^{-c_3\|w_0\|_{H^\alpha}^2n\varepsilon_n^2}\Pi\left(w:\|w\|_\infty\le \frac{1}{2}\sqrt n\varepsilon_n^2\right),
\end{align*}
for some $c_3>0$. An application of the centred small ball inequality in Theorem 1.2 of \cite{LL99}, jointly with the aforementioned complexity bound from Theorem 4.3.36 of \cite{GN16}, then shows that the last display is greater than 
$$
e^{-c_3\|w_0\|_{H^\alpha}^2n\varepsilon_n^2}e^{c_4 (\sqrt n \varepsilon_n^2)^{d/(\alpha-d/2)}}
=e^{-c_5\|w_0\|_{H^\alpha}^2n\varepsilon_n^2}
$$ 
for some $c_4,c_5>0$, concluding the proof. \qed

%
%
%

%%%%%%%%%%%%%%%%%%%%%%%%%%%%%%%%%%%%%%%%%%%%%%%%%%%%%%%%%%%%%%%%%%%%%%%%%
\subsection{Proof of Theorem \ref{Theo:LaplRates}}
\label{Subsec:ProofLaplRates}

The proof follows along similar lines to the one of Theorem \ref{Theo:GPRates}, replacing the Gaussian probability measure techniques employed therein with the corresponding results for (rescaled) Besov-Laplace priors. For $W$ as in \eqref{Eq:BaseLapalPrior}, let $\tilde \Pi_n$ be the law of the random function $\tilde W_n:= (n\varepsilon_n^2)^{-1}W$. Then, for any $c_1>0$, Lemma 9 in \cite{G23} (and its proof) and the Besov embedding $B^{1+d+\kappa}_1([0,1]^d)\subset C^1([0,1]^d)$, holding for all $\kappa>0$, e.g.,~\cite[p.~370]{GN16}, jointly imply that the sets
$$
\Ucal_n=\{ u = u_1 + u_2 : \|u_1\|_{L^2}\le \varepsilon_n,
\|u_2\|_{B^\alpha_{11}}\le c_2, \|u\|_{C^1}\le c_2\}
$$
satisfy $\tilde\Pi_n(\Ucal_n^c)\le e^{-c_1n\varepsilon_n^2}$ provided that $c_2$ is large enough. Arguing exactly as in the proof of Theorem \ref{Theo:GPRates}, the sets
$\Rcal_n=\left\{\eta\circ u,\ u\in \Ucal_n\right\}$ are then seen to verify the conditions of Theorem \ref{Theo:GenTheo}. Further, recalling that $w_0\in B^\beta_1([0,1]^d)\subseteq B^\alpha_1([0,1]^d)$ since $\alpha\le\beta$, an application of the `de-centring' inequality (32) in \cite{G23} and the $L^\infty$-small ball lower bound (34) in the same reference yield
\begin{align*}
	\Pi_n(\rho:\|\rho - \rho_0\|_{L^\infty} \le \varepsilon_n)
	&\ge \tilde\Pi_n\left(w:\|w - w_0\|_{L^\infty} 
	\le \frac{1}{2}\varepsilon_n\right) \\
    & \ge e^{-\|w_0\|_{B_1^\alpha}^2n\varepsilon_n^2}
	\Pi\left(w:\|w\|_\infty\le \frac{1}{2}n\varepsilon_n^3\right)\\
	&\ge e^{-\|w_0\|_{B_1^\alpha}^2n\varepsilon_n^2}
	e^{-c_1(n\varepsilon_n^3)^{d/(\alpha - d)}}
	=e^{-c_2n\varepsilon_n^2}
\end{align*}
for some $c_1,c_2>0$, concluding the proof.\qed

%
%
%
%
%

%%%%%%%%%%%%%%%%%%%%%%%%%%%%%%%%%%%%%%%%%%%%%%%%%%%%%%%%%%%%%%%%%%%%%%%%%%%
\section{Posterior sampling algorithms}
\label{App:Algorithms}

For the considered Gaussian and Besov-Laplace priors, the posterior arising from the observation model \eqref{Eq:PointProc} is not available in closed form. We then employ MCMC techniques to draw approximate posterior samples, concretely computing posterior means and credible sets via their MCMC counterparts. 

%
%
%

%%%%%%%%%%%%%%%%%%%%%%%%%%%%%%%%%%%%%%%%%%%%%%%%%%%%%%%%%%%%%%%%%%%%%%%%%%%
\subsection{The pCN algorithm}
\label{App:pCN}

In the case of the rescaled Gaussian priors $\Pi_n$ from Section \ref{Subsec:GPRates}, we resort to the pre-conditioned Crank-Nicholson (pCN) algorithm, which is a dimension-robust MCMC technique of Metropolis-Hastings type first developed by \cite{CSRW13}. Starting from some initialisation $\omega_0$ (which we set equal to the `cold start' $\omega_0 = 0$ throughout our experiments), the method iterates the following steps:
\begin{itemize}
	\item Draw a sample $\xi\sim \Pi$ from the base Gaussian prior, and construct the rescaled proposal $\omega_*:= \sqrt{1-2b}$ $\omega_{s-1} + \sqrt{2b} \ n^{-d/(4\alpha+2d)} \xi$, where $b\in(0,1/2)$ is a fixed step-size.
	\item Set
	$$
	\omega_{s}:=
	\begin{cases}
		\omega_*, & \textnormal{with probability}\ \min\left\{1, \frac{L^{(n)}(\eta\circ\omega_*)}{L^{(n)}(\eta\circ\omega_{s-1})}\right\},\\
		\omega_{s-1}, & \text{otherwise},
	\end{cases}
	$$
\end{itemize}
with $L^{(n)}$ the likelihood from \eqref{Eq:Likelihood} and $\eta$ the link function selected within the prior specification. The resulting Markov chain is reversible with respect to the posterior distribution, e.g.,~\cite[Proposition 1.2.2]{N23}. Further, the pCN acceptance probabilities are generally known to be stable with respect to the discretisation resolution, yielding favourable mixing properties for statistical applications with infinite-dimensional unknowns; see \cite{CSRW13}.

The first pCN step requires drawing $\xi\sim\Pi$, which is typically straightforward under a suitable discretisation scheme. For example, for the Gaussian wavelet series described in Example \ref{Ex:GPWav}, this naturally amounts to truncating the expansion \eqref{Eq:GPWav} at some sufficiently high level $L\in\N$ and then sampling $L$ independent standard normal random variables. For the Matérn process priors from Example \ref{Ex:Matern}, a discretisation scheme based on piece-wise linear interpolation functions over a pre-specified grid on $[0,1]^d$ can be employed, as described e.g.,~in \cite[p.~670]{MOLL07}.

The second operation requires evaluation of the proposal likelihood. This is analytically intractable, as it involves integration of the associated spatial intensity over $\Wcal_n$, which is generally not available in closed form. To address this issue, we employ standard numerical integration; specifically, piece-wise polynomial quadrature. In the prototypical spatial statistics applications of covariate-driven point processes, such as the ones developed in Section \ref{Sec:RealData} and Appendix \ref{Suppl:RealData}, where the number of covariates is relatively low, we found this approach to strike a satisfactory balance between accuracy and efficiency. In our computations, each instance of numerical integration was implemented based on rectangular grids of up to 2,500 nodes, depending on the area of the observation window. For models without covariates, we refer to \cite{AMM09} for more refined `exact' MCMC methods based on data augmentation. Extending these to the present setting is an interesting open problem.

%
%
%

%%%%%%%%%%%%%%%%%%%%%%%%%%%%%%%%%%%%%%%%%%%%%%%%%%%%%%%%%%%%%%%%%%%%%%%%%%%
\subsection{The whitened wpCN algorithm}
\label{App:wpCN}

The whitened pCN (wpCN) algorithm \cite{CDPS18} is an extension of the base method described above that is suited to priors that can be expressed as a transformation of a Gaussian white noise. For the Besov-Laplace prior from Section \ref{Subsec:LaplRates}, this is based on the observation that the rescaled random function $n^{-d/(2\alpha+d)}W$ from \eqref{Eq:RescaledLaplPrior} is equal in distribution to
\begin{equation}
	\label{Eq:WhiteXi}   
	T^{(n)}(\xi)(z) := \sum_{\ell=1}^\infty T^{(n)}_\ell(w_\ell) \psi_\ell(z), 
	\qquad z\in[0,1]^d,
	\qquad w_\ell\iid N(0,1),
\end{equation}
where 
\begin{equation}
	\label{Eq:Xi} 
	\xi := \sum_{\ell=1}^\infty w_\ell \psi_\ell, 
	\qquad w_\ell\iid N(0,1),
\end{equation}
defines a Gaussian white noise indexed by $[0,1]^d$ and the `whitening transformation' $T$ is given by
$$
T^{(n)}_\ell(w_\ell) := n^{-\frac{d}{2\alpha+d}}\ell^{-(\frac{\alpha}{d}-\frac{1}{2})} \sgn(w_\ell) \left[ -\ln(2 - 2\phi(|w_\ell|) \right],
\qquad \ell\in\N.
$$

Then, starting from an initial white noise sample $\xi_0$ (amounting to drawing independent standard normal random variables), and the corresponding transformed sample $\omega_0 = T^{(n)}(\xi_0)$, the wpCN algorithm repeats the following three operations:
\begin{itemize}
	\item Construct the whitened proposal $\xi_* :=\sqrt{1-2b}$ $\xi_{s-1} + \sqrt{2b} \zeta$, where $b\in(0,1/2)$ is a fixed step-size and $\zeta$ is an independent Gaussian white noise.
	\item Set
	$$
	\xi_{s}:=
	\begin{cases}
		\xi, & \textnormal{with probability}\ \min\left\{1, \frac{L^{(n)}(\eta\circ T^{(n)}(\xi_*))}
		{L^{(n)}(\eta\circ \omega_{s-1})}\right\},\\
		\xi_{s-1}, & \text{otherwise}.
	\end{cases}
	$$
	\item Set $\omega_s := T^{(n)}(\xi_s)$.
\end{itemize}
The first step can be straightforwardly implemented by truncating the series in \eqref{Eq:WhiteXi} and \eqref{Eq:Xi} at some pre-specified level $L\in\N$. The second operation necessitates the evaluation of the proposal likelihood, for which we again employ numerical integration, similar to the pCN method described above.

%
%
%

%%%%%%%%%%%%%%%%%%%%%%%%%%%%%%%%%%%%%%%%%%%%%%%%%%%%%%%%%%%%%%%%%%%%%%%%%%%
\subsection{Metropolis-within-Gibbs algorithms}
\label{App:AdaptiveMCMC}

For the implementation of the hierarchical Gaussian wavelet series prior and hierarchical Besov-Laplace priors considered in Section \ref{Subsec:Adaptive}, where $\alpha\sim\text{Exp}(1)$, each iteration of the above pCN and wpCN algorithms must be augmented with an additional MCMC-based sampling step to account for the uncertainty in the hyper-parameter $\alpha$. Starting from an initial white noise sample $\xi_0$ and the initialiser $\alpha_0>0$ for $\alpha$, the resulting  Metropolis-within-Gibbs sampling algorithms then iterate the following steps:
\begin{itemize}
	\item Draw the whitened update $\xi_s$ as in the wpCN algorithm from \ref{App:wpCN}, where for Besov-Laplace priors, $T^{(n)}=T^{(n)}_{\alpha_{s-1}}$ is given by \eqref{Eq:WhiteXi} with $\alpha = \alpha_{s-1}$, while for Gaussian priors it is given by $T^{(n)}_{\alpha_{s-1},\ell}(w_\ell) = n^{-\frac{d}{4\alpha+2d}}\ell^{-\frac{\alpha}{d}}w_\ell$. Set $\omega_{s} = T^{(n)}_{\alpha_{s-1}}(\xi_s)$.
	
	\item Draw the update $\alpha_s$ via a simple random-walk Metropolis-Hastings procedure:
	\begin{itemize}
		\item Construct the proposal $\alpha_* := \max\{\alpha_{s-1} + c Z,0\}$, where $c > 0$ is a fixed step-size and $Z$ is an independent standard Gaussian random variable.
		\item Set
		$$
		\alpha_s:=
		\begin{cases}
			\alpha_*, & \textnormal{with probability}\ 
			\min\left\{1, \frac{L^{(n)}(\eta \circ T^{(n)}_{\alpha^*}(\xi_s))}
			{L^{(n)}(\eta \circ T^{(n)}_{\alpha_{s-1}}(\xi_s))} 
			\times \frac{\pi( \alpha_*)}{\pi(\alpha_{s-1})}\right\},\\
			\alpha_{s-1}, & \text{otherwise},
		\end{cases}
		$$
		with $\pi$ the p.d.f.~of the hyper-prior of $\alpha$.
	\end{itemize}
\end{itemize}

Given $\alpha_{s-1}$, the update of the high-dimensional parameters $\xi_s$ and $\omega_s$ thus remains essentially unchanged with respect to the pCN and wpCN algorithms, with the exception that the regularity-dependent transformation $T^{(n)}_{\alpha_{s-1}}$ now also varies throughout the run. The update to $\alpha_{s}$ involves the additional likelihood evaluation $L^{(n)}(\eta \circ T^{(n)}_{\alpha^*}(\xi_s))$. This is tackled via numerical integration, so that the overall cost of the second updating step is comparable to the first. Per iteration run times for the Metropolis-within-Gibbs algorithm are thus roughly double the ones for the base pCN and wpCN routines, if all other computational tuning parameters (e.g.,~number of nodes for numerical integration) are left unchanged.

\section{Additional simulation studies}
	\label{Suppl:AddSimul}
	
	We expand the simulation studies presented in Section \ref{Sec:Simulations}, considering bi-dimensional scenarios with both spatially homogeneous and inhomogeneous ground truths.
	
	With $\Wcal_n$ as in \eqref{Eq:SpatialWn}, we simulate bivariate covariates $Z^{(n)}(x) = (Z_1^{(n)}(x),$ $Z_2^{(n)}(x))$, $x\in\Wcal_n$, taking $Z_1^{(n)}:=(Z_1^{(n)}(x),\ x \in \Wcal_n)$ as in the univariate experiments, and setting $Z_2^{(n)}:=(Z_2^{(n)}(x),\ x \in \Wcal_n)$ equal to an independent (centred and unit variance) square-exponential Gaussian process with larger length-scale 1.5, transformed in accordance with Example \ref{Ex:GaussCov}. We then consider two ground truths defined on the covariate space $[0,1]^2$, respectively:
	\begin{enumerate}
		\item A spatially homogeneous covariate-based intensity function,
		\begin{align}
			\label{Eq:2D_skn}
			\rho_0(z_1,z_2) = 100 f_{SN}\left(z_1,z_2; \ (0.4, 0.6), \ 0.05 I_2, \ (3,-2) \right),
		\end{align}
		with $(z_1,z_2)\in[0,1]^2$, where $f_{SN}$ denotes the (bivariate) skew-normal p.d.f.~and $I_2$ is the identity matrix in $\R^{2,2}$, cf.~Figure \ref{Fig:2D_skn};

		\item A spatially inhomogeneous function with a bi-dimensional blocky component and a localised spike, 
		\begin{equation}
			\label{Eq:2D_bs}
			\rho_0(z) = \frac{100}{0.42} [ a S_{\text{block}}(z; b, c) + h S_{\text{spike}}(z; u, w) ],
		\end{equation}
		with $z=(z_1,z_2)\in[0,1]^2$, height $h = 20$, amplitude $a = 4$, and where the two components are defined by
		\begin{align*}
			S_{\text{block}}(z; b, c) & = \frac{1}{2} \prod_{h = 1}^2 (1 + \sgn(z_h - b_h)) (1 - \sgn(z_h - c_h));\\
			S_{\text{spike}}(z; u, w) & = (1 + |(z-u)/w|)^{-4},
		\end{align*}
		with block extremes $b = (0.1, 0.2)$ and $c = (0.3, 0.5)$, spike location $u=(0.7,0.8)$,  and width $w = 0.1$. See the last panel of Figure \ref{Fig:2D_bs}.
	\end{enumerate}

	Both the above ground truths have been rescaled so that the expected number of points per unit area within the observation model \eqref{Eq:PointProc} is equal to $100$. For each scenario, we proceed computing the posterior means $\hat\rho^{(n)}_\Pi$ associated to Gaussian wavelet series and Besov-Laplace priors, employing bi-dimensional, boundary reflected, `Symmelet-8' basis functions (implemented in the $\texttt{R}$ package $\texttt{wavethresh}$) and the same oracle tuning step for the prior regularity as the one taken in the one-dimensional experiments. Bivariate kernel estimates $\hat\rho^{(n)}_\kappa$ were also obtained (through the off-the-shelf routine included in the $\texttt{R}$ package $\texttt{spatstat}$ \cite{BRT16}).

	The results for the spatially homogeneous ground truth \eqref{Eq:2D_skn}, averaged over 50 replications of each experiment, are summarised in Table \ref{Tab:2D_skn}. For all three considered methods, the obtained estimation errors monotonically decay as the area increases. In particular, Gaussian and Besov-Laplace priors performed similarly, as expected in the spatially homogeneous case in view of Theorems \ref{Theo:GPRates} and \ref{Theo:LaplRates}. The kernel method also achieved satisfactory reconstructions, but overall scored higher estimation errors across all values of $n$ except for the smallest one, $n=1$. A depiction of the posterior means for the employed Gaussian series priors is provided in the first three panels of Figure \ref{Fig:2D_skn}, where a progressive improvement in the visual agreement with the ground truth \eqref{Eq:2D_skn} (shown in the last panel) is displayed.

	\begin{table}[H]%%%%%%%%%%%%%%%%%%%%%%%%%%%%%%%%%%%%%%%%%%%%%%%%%%%%%%%%%%%%
		\caption{Averaged relative $L^1$-estimation errors 
			$\|\hat\rho-\rho_0\|_{L^1}/\|\rho_0\|_{L^1}$ (with standard deviations in parentheses) over $50$ replications in the bivariate homogeneous scenario with ground truth \eqref{Eq:2D_skn}. The table compares Gaussian wavelet series and Besov--Laplace priors (with oracle-tuned regularity parameters $\alpha=1.5$ and $\alpha=1$, respectively) against a kernel-based estimator. For reference, $\|\rho_0\|_{L^1}=100$.}
		\centering
		\hrule
		\begin{tabular}{cccc}
			\makecell{\\
				$n$} & 
			\makecell{Gaussian ($\alpha = 1.5$)\\
				$\frac{\|\hat\rho_\Pi^{(n)} - \rho_0\|_{L^1}}{\|\rho_0\|_{L^1}}$}
			& 
			\makecell{Laplace ($\alpha = 1$)\\
				$\frac{\|\hat\rho_\Pi^{(n)} - \rho_0\|_{L^1}}{\|\rho_0\|_{L^1}}$}
			& \makecell{\\
				$\frac{\|\hat\rho_\kappa^{(n)} - \rho_0\|_{L^1}}{\|\rho_0\|_{L^1}}$} \\
			1   & 0.54 (0.23) & 0.50 (0.08) & 0.52 (0.18) \\
			4   & 0.33 (0.03) & 0.29 (0.04) & 0.42 (0.14) \\
			16  & 0.21 (0.03) & 0.19 (0.02) & 0.31 (0.04) \\
			64  & 0.14 (0.01) & 0.13 (0.01) & 0.28 (0.02) \\
			256 & 0.10 (0.01) & 0.10 (0.01) & 0.26 (0.01) \\
		\end{tabular}
		\hrule
		\label{Tab:2D_skn}
	\end{table}%%%%%%%%%%%%%%%%%%%%%%%%%%%%%%%%%%%%%%%%%%%%%%%%%%%%%%%%%%%%%%%%%

	\begin{figure}[H]%%%%%%%%%%%%%%%%%%%%%%%%%%%%%%%%%%%%%%%%%%%%%%%%%%%%%%%%%%%
		\centering
		
		\includegraphics[width=\linewidth]{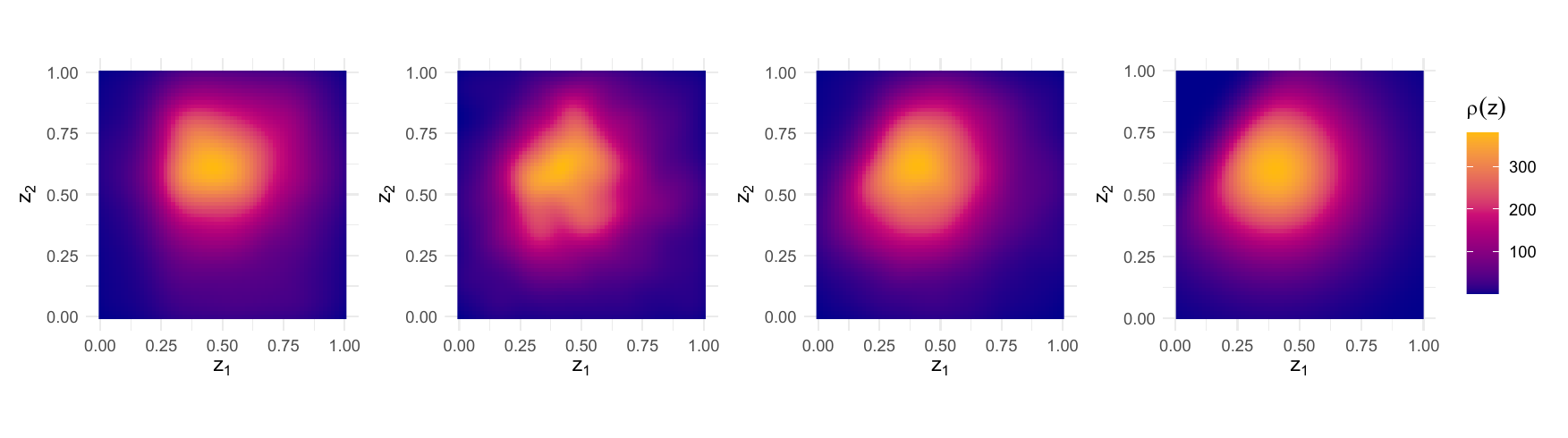}
		\caption{Posterior mean reconstructions for the Gaussian wavelet series prior in the bivariate homogeneous setting \eqref{Eq:2D_skn}, shown for increasing observation windows $n\in\{4,16,256\}$, alongside the true intensity (last panel).}
		\label{Fig:2D_skn}
	\end{figure}%%%%%%%%%%%%%%%%%%%%%%%%%%%%%%%%%%%%%%%%%%%%%%%%%%%%%%%%%%%%%%%%

	Table \ref{Tab:2D_bs} reports the average $L^1$-estimation errors obtained in the bivariate spatially inhomogeneous scenario with covariate-based intensity function \eqref{Eq:2D_bs}. Similar to the experiments presented in Section \ref{Subsec:1DInhom}, the kernel method appears to be unsuited to recover the localised structural features of the ground truth in this case. Across all values of $n$, the best performance was scored by the Besov-Laplace prior, with particularly robust improvements over the Gaussian prior for larger areas. These results are visualised in Figure \ref{Fig:2D_bs}, where the posterior means for the Besov-Laplace prior are seen to progressively better recover both the blocky component and the spike of the ground truth \eqref{Eq:2D_bs} (shown in the last panel).

	\begin{table}[H]%%%%%%%%%%%%%%%%%%%%%%%%%%%%%%%%%%%%%%%%%%%%%%%%%%%%%%%%%%%%
		\centering
		\caption{Averaged relative $L^1$-estimation errors 
			$\|\hat\rho-\rho_0\|_{L^1}/\|\rho_0\|_{L^1}$ (with standard deviations in parentheses) over $50$ replications in the bivariate inhomogeneous scenario with ground truth \eqref{Eq:2D_bs}. The comparison includes Gaussian and Besov--Laplace wavelet priors (both using oracle-tuned regularity parameter $\alpha=1$) and a kernel estimator. For reference, $\|\rho_0\|_{L^1}=100$.}
		\hrule
		\begin{tabular}{cccc}
			\makecell{\\
				$n$} & 
			\makecell{Gaussian ($\alpha = 1$)\\
				$\frac{\|\hat\rho_\Pi^{(n)} - \rho_0\|_{L^1}}{\|\rho_0\|_{L^1}}$}
			& 
			\makecell{Laplace ($\alpha = 1$)\\
				$\frac{\|\hat\rho_\Pi^{(n)} - \rho_0\|_{L^1}}{\|\rho_0\|_{L^1}}$}
			& \makecell{\\
				$\frac{\|\hat\rho_\kappa^{(n)} - \rho_0\|_{L^1}}{\|\rho_0\|_{L^1}}$} \\
			1   & 0.88 (0.14) & 0.80 (0.16) & 0.87 (0.09) \\
			4   & 0.73 (0.09) & 0.72 (0.24) & 0.85 (0.06) \\
			16  & 0.62 (0.09) & 0.60 (0.14) & 0.81 (0.04) \\
			64  & 0.59 (0.08) & 0.49 (0.09) & 0.77 (0.01) \\
			256 & 0.45 (0.05) & 0.39 (0.01) & 0.77 (0.01) \\
		\end{tabular}
		\hrule
		\label{Tab:2D_bs}
	\end{table}%%%%%%%%%%%%%%%%%%%%%%%%%%%%%%%%%%%%%%%%%%%%%%%%%%%%%%%%%%%%%%%%%

	\begin{figure}[H]%%%%%%%%%%%%%%%%%%%%%%%%%%%%%%%%%%%%%%%%%%%%%%%%%%%%%%%%%%%
		\centering
		\includegraphics[width=\linewidth]{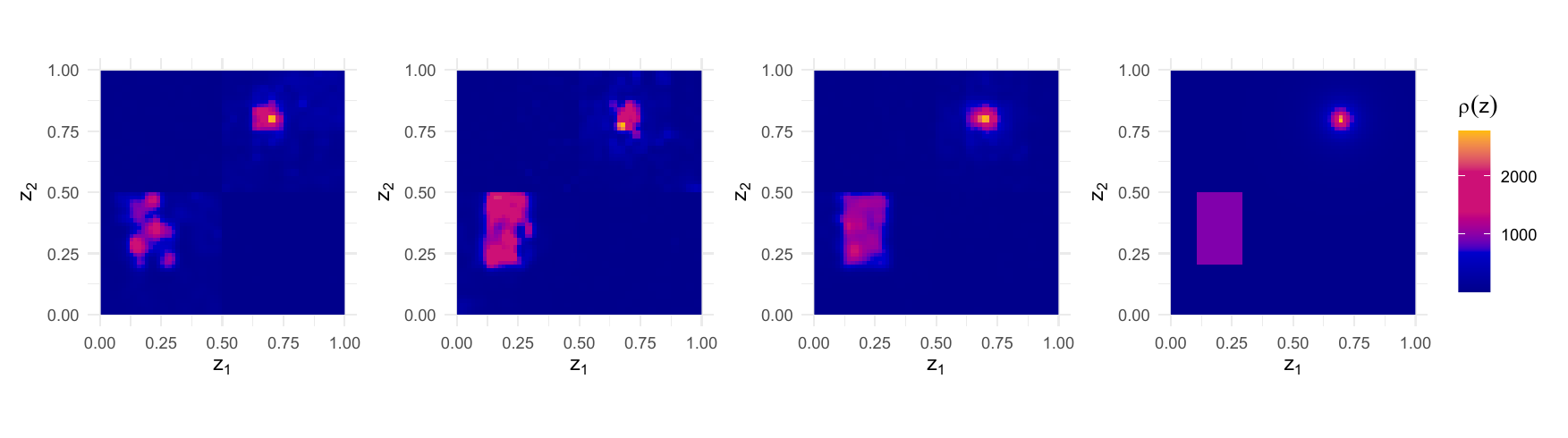}
		\caption{Posterior mean reconstructions under the Besov--Laplace prior for the bivariate inhomogeneous intensity \eqref{Eq:2D_bs}, shown for increasing observation windows $n\in\{4,16,256\}$, together with the true intensity (last panel).}
		\label{Fig:2D_bs}
	\end{figure}%%%%%%%%%%%%%%%%%%%%%%%%%%%%%%%%%%%%%%%%%%%%%%%%%%%%%%%%%%%%%%%%

	%%%%%%%%%%%%%%%%%%%%%%%%%%%%%%%%%%%%%%%%%%%%%%%%%%%%%%%%%%%%%%%%%%%%%%
	\section{Applications to a Canadian wildfire dataset}
	\label{Suppl:RealData}
	
	Forest fires are frequent and severe natural disturbances in Canada. A powerful monitoring system has long been maintained, with extensive data on the wildfire locations, their severity and extension, as well as detailed environmental information publicly available at the Canadian Wildland Fire Information System website (\url{http://cwfis.cfs.nrcan.gc.ca/home}). Several studies have highlighted that the wildfire activity is predominantly influenced by meteorological conditions such as long periods without rain, high temperatures and strong winds; see \cite{JMS12,BGMMM20,KPDO23}, among others.

	Here, we examine a dataset, previously investigated by \cite{BGMMM20} using kernel methods, comprising the aggregate hotspots over the month of June 2015 (which is within the peak of the Canadian wildfire season, lasting from late April to August), and location-specific temperatures and precipitation levels. Specifically, to avoid extremes, the third quartile of the temperatures registered in June 2015 and the average daily precipitations are used. The data is visualised in Figure \ref{Fig:can_fires}. Based on these, we quantitatively assess the influence of temperatures and precipitation on the intensity of wildfires via covariate-based nonparametric Bayesian intensity estimation. As for the real data analysis from Section \ref{Sec:RealData}, we focus on Besov-Laplace priors to better capture potentially localised features in the covariate-based intensity function and the point pattern, such as the distinct triangular region with little to no hotspots in the central-southern part of Canada, which is mostly covered by prairies and croplands, and is therefore significantly less subject to wildfires.

	\begin{figure}[H]%%%%%%%%%%%%%%%%%%%%%%%%%%%%%%%%%%%%%%%%%%%%%%%%%%%%%%%%%%%
		\centering
		\includegraphics[width=\linewidth]{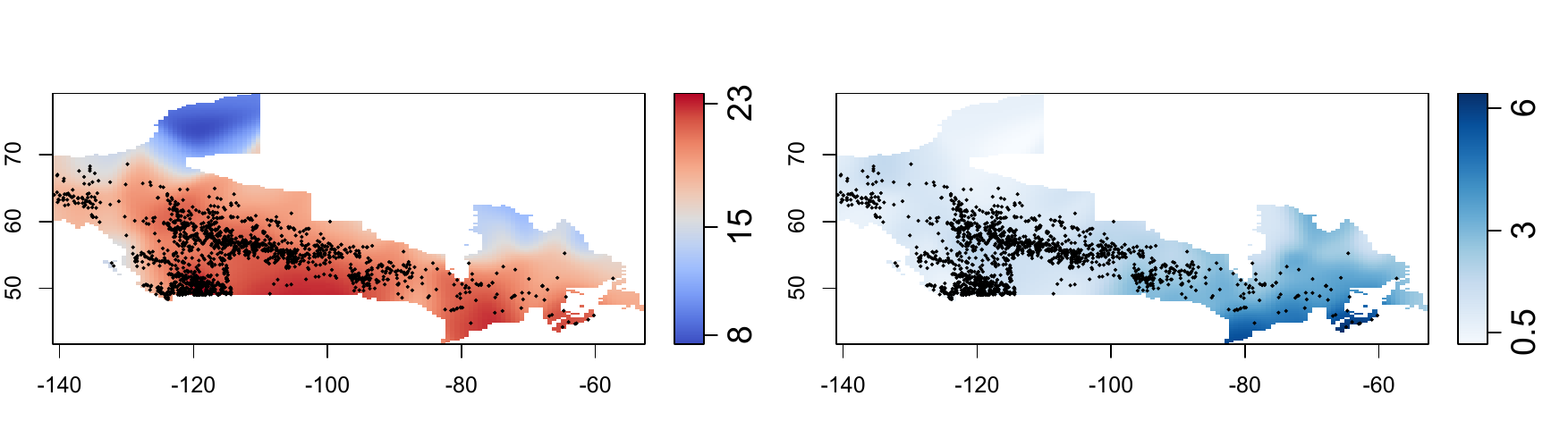}
		\caption{Environmental covariates and wildfire locations in Canada for June 2015. The left panel displays the third quartile of daily temperatures (in $^\circ$C), while the right panel shows the average daily precipitation (in mm/m$^2$) over the observation window. Black points indicate recorded wildfire hotspots. The spatial distribution of points suggests a pronounced dependence of wildfire activity on both temperature and precipitation, with clear clustering in warm and dry regions and a large low-activity area in central-southern Canada.}
		\label{Fig:can_fires}
	\end{figure}%%%%%%%%%%%%%%%%%%%%%%%%%%%%%%%%%%%%%%%%%%%%%%%%%%%%%%%%%%%%%%%%

	We start investigating the influence of each covariate individually. The posterior means of the temperature- and precipitation-based intensity function are shown in Figure \ref{Fig:canada_rho_1d}, alongside benchmark kernel estimates. In line with the literature, these detect a positive association between higher temperatures and increased risk of wildfires, with a sharp rise between $17^\circ$C and $22^\circ$C. Both methods detect a prominent local minimum around $23^\circ$C, which is possibly a spurious effect determined by the high temperatures registered in the central-southern flat region. Precipitation levels were found to have a strong negative impact, particularly above 1.5 $\text{mm/m}^2$.

	\begin{figure}[H]%%%%%%%%%%%%%%%%%%%%%%%%%%%%%%%%%%%%%%%%%%%%%%%%%%%%%%%%%%%
		\centering
		\includegraphics[width=\linewidth]{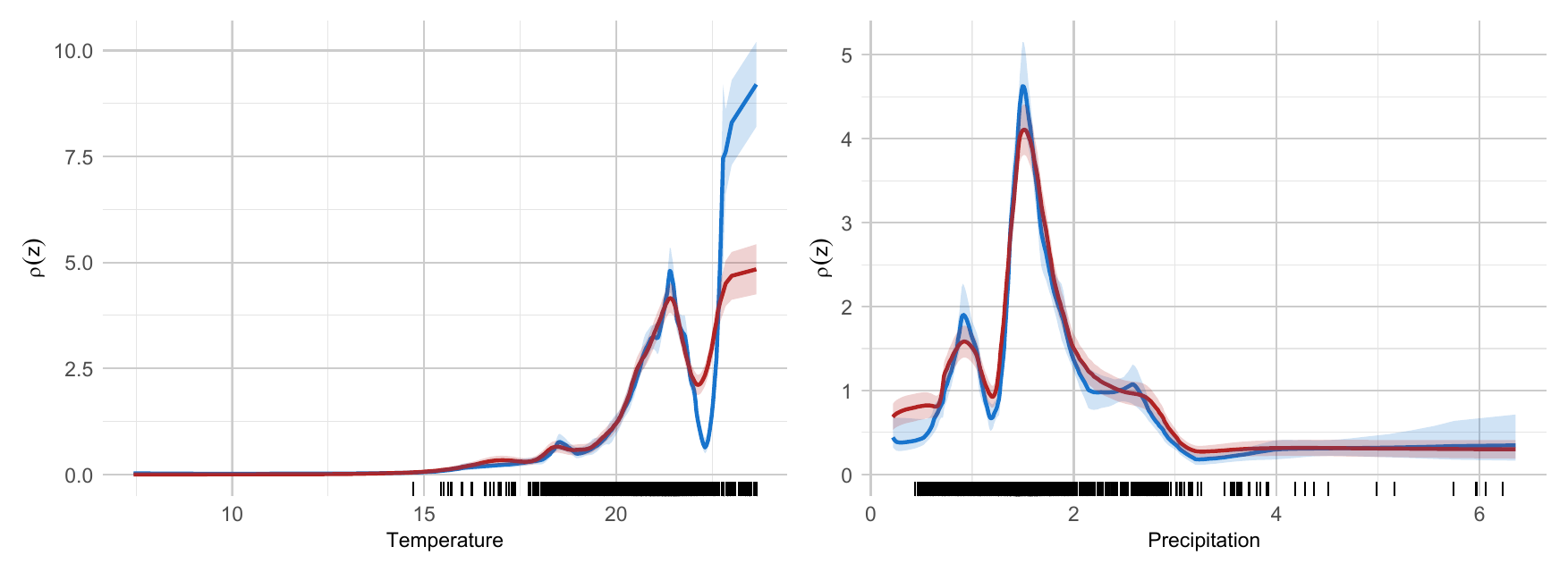}
		\caption{Estimated covariate-dependent wildfire intensity as a function of temperature (left) and precipitation (right). The solid blue curves represent posterior means under the Besov--Laplace prior with regularity parameter $\alpha=1$, while the shaded regions denote pointwise $95\%$ credible intervals. Kernel-based estimates from \cite{BGMMM20} are shown in red. The results indicate a positive association between temperature and wildfire intensity, with a marked increase between $17^\circ$C and $22^\circ$C, and a strong negative effect of precipitation, particularly beyond approximately $1.5$ mm/m$^2$. A local minimum around $23^\circ$C is detected by both methods and may reflect spurious variability induced by spatial heterogeneity in the data.}
		\label{Fig:canada_rho_1d}
	\end{figure}%%%%%%%%%%%%%%%%%%%%%%%%%%%%%%%%%%%%%%%%%%%%%%%%%%%%%%%%%%%%%%%

	We proceed fitting the full model based on both meteorological covariates. The posterior mean and kernel estimate are reported in Figure \ref{Fig:canada_rho_2d}. They both identify higher intensities in the top left corner of the covariate space, corresponding to warmer and drier conditions. Compared to the kernel method, the result obtained with the Besov-Laplace prior detects overall stronger intensities, and notably sharper transitions between the high-intensity zones from the low-intensity ones. Further, the posterior mean also clearly chargers low values to the localised region with temperatures above $22^\circ$C and precipitation levels below 1 mm/$\text{m}^2$, which characterise the triangular central-southern area not subject to wildfires in Figure \ref{Fig:can_fires}.

	\begin{figure}[H]%%%%%%%%%%%%%%%%%%%%%%%%%%%%%%%%%%%%%%%%%%%%%%%%%%%%%%%%%%%
		\centering
		\includegraphics[width=\linewidth]{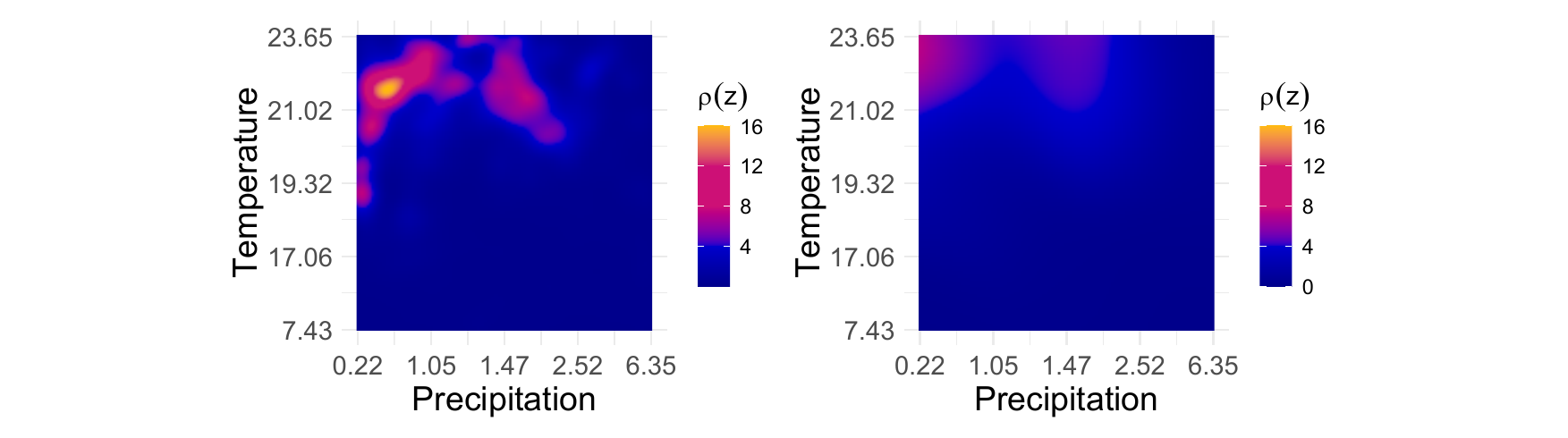}
		\caption{Estimated wildfire intensity as a joint function of temperature and precipitation. The left panel shows the posterior mean under the Besov--Laplace prior with regularity parameter $\alpha=1$, while the right panel reports the corresponding kernel estimate. Both methods identify higher intensities in warm and dry regions (upper-left corner of the covariate space), but the Besov--Laplace estimator produces sharper transitions between high- and low-intensity regimes. In particular, it more clearly isolates the low-intensity region associated with temperatures above $22^\circ$C and precipitation below $1$ mm/m$^2$, corresponding to the central-southern region of Canada with scarce wildfire activity.}
		\label{Fig:canada_rho_2d}
	\end{figure}%%%%%%%%%%%%%%%%%%%%%%%%%%%%%%%%%%%%%%%%%%%%%%%%%%%%%%%%%%%%%%%

	Lastly, we display the corresponding plug-in estimates of the spatial intensity function $\lambda^{(n)}_\rho$ from \eqref{Eq:PointProc} in Figure \ref{Fig:canada_spatial2}. Here, the aforementioned localised features of the posterior mean yield a precise reconstruction of the point pattern (top left plot), with a clear detection of the central-southern flat region, over which extremely low intensity is inferred. The estimated maximum intensity is located at the south-western corner and along the central band north of the Great Plains, which indeed are characterised by the highest density of wildfires. The plug-in kernel estimate (bottom right plot) was unable to produce such detailed reconstruction, rather yielding a much more diffuse spread of the spatial intensity. In the bottom left corner of Figure \ref{Fig:canada_spatial2}, we include a further benchmark obtained via the `bootstrapped' kernel procedure for covariate-based spatial intensity estimation developed by \cite{BGMMM20}. Consistently with the empirical results from the latter reference, this improves on the standard kernel method, achieving a more faithful reconstruction of the point pattern that is similar in shape to the one obtained via Besov-Laplace priors, albeit with less precisely separated zones of high and low intensity.
	
	\begin{figure}[H]%%%%%%%%%%%%%%%%%%%%%%%%%%%%%%%%%%%%%%%%%%%%%%%%%%%%%%%%%%
		\centering
		\includegraphics[width=\linewidth]{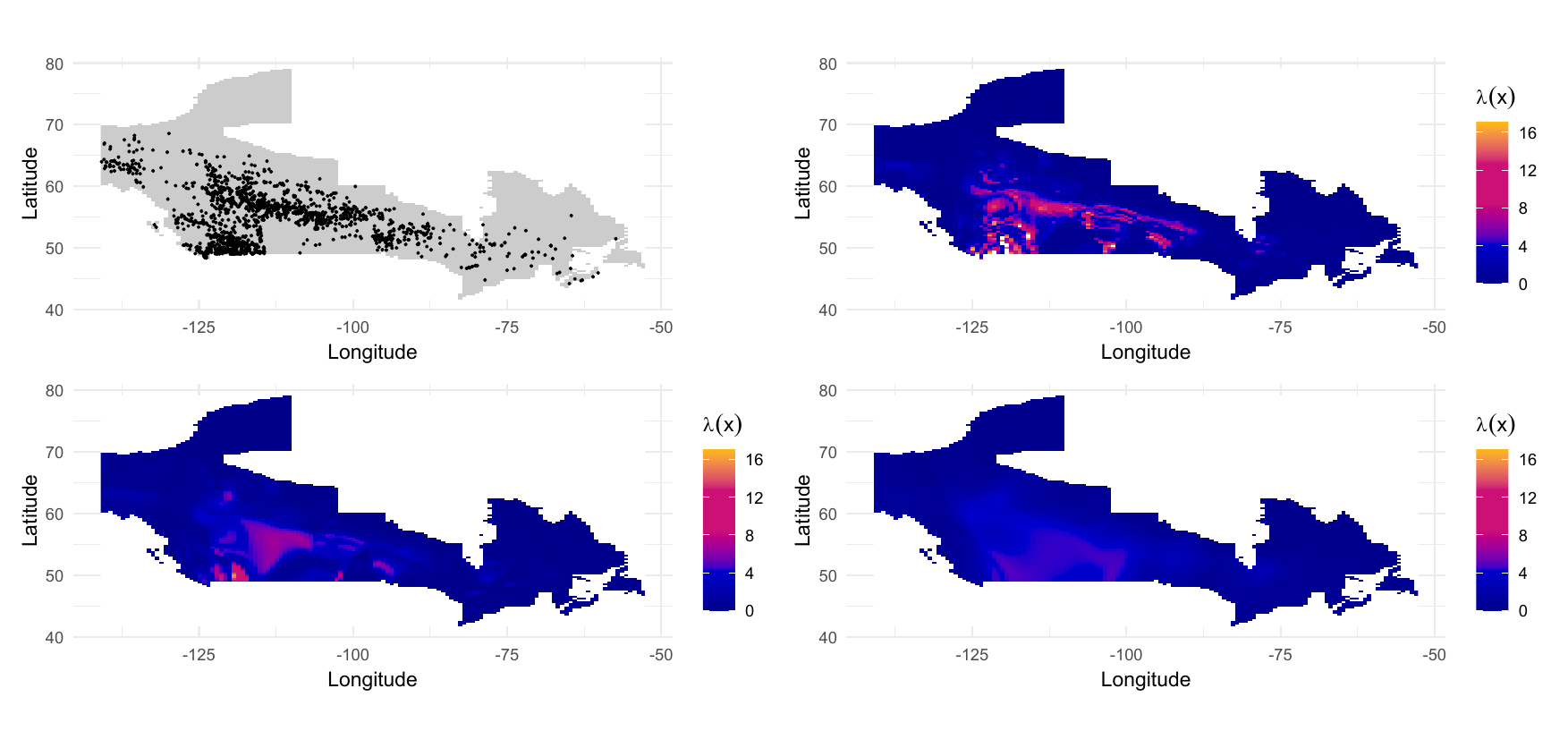}
		\caption{Plug-in estimates of the spatial intensity function for Canadian wildfires in June 2015. The top-left panel shows observed wildfire locations. The top-right panel displays the plug-in posterior mean based on the Besov--Laplace prior, while the bottom-left and bottom-right panels report the bootstrapped kernel estimator from \cite{BGMMM20} and the standard kernel estimator, respectively. }
		\label{Fig:canada_spatial2}
	\end{figure}%%%%%%%%%%%%%%%%%%%%%%%%%%%%%%%%%%%%%%%%%%%%%%%%%%%%%%%%%%%%%%%%

\end{document}